\documentclass[12pt, oneside, reqno]{amsart}
\usepackage{graphicx} 
\usepackage{amssymb}
\usepackage{amsthm}
\usepackage{amsmath}
\usepackage[margin=1.5cm]{geometry}
\usepackage{pdfpages}
\usepackage[all]{xy}
\usepackage{comment}

\usepackage{geometry}
\usepackage[colorlinks = true,linkcolor = blue,linktocpage=true]{hyperref}
\newtheorem{theorem}{Theorem}[section]
\newtheorem{lemma}[theorem]{Lemma}
\newtheorem{example}[theorem]{Example}
\newtheorem{corollary}[theorem]{Corollary}
\newtheorem{definition}[theorem]{Definition}
\newtheorem{convention}[theorem]{Convention}
\newtheorem{proposition}[theorem]{Proposition}
\newtheorem{maintheorem}{Theorem}
\newtheorem{remark}[theorem]{Remark}
\newtheorem*{question}{Question}

\title[Severi-Brauer surfaces]{A determinant on birational maps of Severi-Brauer surfaces}
\author{Elias Kurz}
\date{\today}
\address{Elias Kurz, Universit\'e de Neuch\^atel,
Institut de Math\'ematiques,
Rue Emile-Argand 11,
Switzerland}

\begin{document}

\maketitle

 \begin{abstract}
     We define a determinant on the group of automorphisms of non-trivial Severi-Brauer surfaces over a perfect field. Using the generators and relations, we extend this determinant to birational maps between Severi-Brauer surfaces. Using this determinant and a group homomorphism found in ~\cite{BSY} we can determine the abelianization of the group of birational transformations of a non-trivial Severi-Brauer surface. This is the first example of an abelianization of the group of birational transformations of a geometrically rational surface where the automorphisms are non trivial. Using the abelianization we find maximal subgroups of the group of birational transformations of a non-trivial Severi-Brauer surface over a perfect field.
 \end{abstract}

\tableofcontents

\newpage

\section{Introduction}
\noindent
Let $S$ be a geometrically rational surface over some field $\textbf{k}$. The group of birational transformations $Bir_{\textbf{k}}(S)$ has been studied a lot. One of the problems related to the group of birational transformations is to find the abelianization, which is the group modulo its commutator subgroup and thus the largest abelian quotient. For an algebraically closed field the Cremona group is generated by $\textrm{PGL}_3(\textbf{k})$ and the standard quadratic involution. Over an algebraically closed field $\textrm{PGL}_3(\textbf{k}) = \textrm{PSL}_3(\textbf{k})$, which is simple, thus this is trivial in the abelianization. Since the standard quadratic involution is conjugated to an automorphism the abelianization is the trivial group. In \cite{Zimmermann_2018} it was shown, that over the real numbers the abelianization of the reel Cremona group $Bir_{\mathbb{R}}(\mathbb{P}^2)$ is given by $\bigoplus_{(0,1]} \mathbb{Z}/2\mathbb{Z}$. In this case, the automorphisms are sent to the trivial element. We also have, that the abelianization is a $2$-torsion group. Over a perfect field $\textbf{k}$ such that $[\bar{\textbf{k}}:\textbf{k}] > 3$ it was shown in \cite{LamySchneider_2024} that the plane Cremona group is generated by involutions and admits a lot of group homomorphisms to $\mathbb{Z} / 2 \mathbb{Z}$, so the abelianization is a highly non-trivial product of copies of $\mathbb{Z}/2 \mathbb{Z}$ (\cite{Lamy_2020}, \cite{LamySchneider_2024}). Again one can show that the automorphisms are trivial in the abelianization.
\newline Let $\textbf{k}$ be a perfect field. A Severi-Brauer surface is a projective smooth surface $S$ over $\textbf{k}$, such that $S_{\bar{\textbf{k}}} \simeq \mathbb{P}^2_{\bar{\textbf{k}}}$. By a theorem from Châtelet we have that $S$ is isomorphic to $\mathbb{P}^2$ already over $\textbf{k}$ if and only if it contains a \textbf{k}-point. In that case we call the Severi-Brauer surface trivial. For a Galois extension $\textbf{k} \subset \textbf{L}$ there is a one-to-one correspondence between isomorphy classes of Severi-Brauer surfaces that split over $\textbf{L}$ and $H^1(Gal(\textbf{L}/ \textbf{k}), \textrm{PGL}_3(\textbf{L}))$ (\cite{serre2004corps}). To find this correspondence one takes an $\textbf{L}$-isomorphism $\varphi$ from $S$ to $\mathbb{P}^2$. After this you can see, that $\varphi g \varphi^{-1} = \alpha_g g, g \in \textrm{Gal}(\textbf{L},\textbf{k})$ for an automorphism $\alpha_g$ of $\mathbb{P}^2$. The map $g \mapsto \alpha_g$ is then a cocycle in $H^1(Gal(\textbf{L}/ \textbf{k}), \textrm{PGL}_3(\textbf{L}))$ using Galois descent one can prove surjectivity. There is also a one-to-one correspondence to central simple algebras over $\textbf{k}$ of dimension $9$ that split over $\textbf{L}$ up to Brauer equivalence, using a similar technique. Using the same central simple algebra but with opposite multiplication we get a Severi-Brauer surface $S^{op}$, which is non-isomorphic to $S$ if $S$ is not trivial. From \cite{kollár2016} we get, that $S$ only contains $d$-points, where $3$ divides $d$ and there exists at least one $3$-point in $S$. The reason why Severi Brauer surfaces are of particular interest are the results from \cite{BSY}, which studies the birational maps of Severi-Brauer surfaces and gives a particularly interesting group homomorphism:
\[Bir_{\textbf{k}}(S) \rightarrow \bigoplus_{p \in \mathcal{E}_3 \setminus \{ q \}} \mathbb{Z}/3\mathbb{Z} \oplus \big (\bigoplus_{p \in \mathcal{E}_6} \mathbb{Z} \big )\]
The $\mathbb{Z}$ part of this group depends on the $6$-points, which do not necessarily always exist. In the case where they exist, we can already see that the abelianization of the group of birational transformations is not a torsion group in contrary to the prior examples. This group homomorphism does not yet give the abelianization. It has in its kernel all the automorphisms. While in the prior examples the automorphisms have always been trivial in the abelianization, we show in this paper, that for non-trivial Severi-Brauer surfaces this is not the case.
\newline 
To find the abelianization of the group of birational transformations of a non-trivial Severi-Brauer surface, we look at a group homomorphism $det: Aut_{\textbf{k}}(S) \rightarrow \textbf{k}^* / (\textbf{k}^*)^3$ and extend this to the birational transformations. To construct this homomorphism, we look at a splitting field $\textbf{L}$ of $S$ and use an isomorphism $\varphi$ from $S$ to the projective plane over $\textbf{L}$. Conjugating the automorphisms of $S$ with this isomorphism, we arrive at an element in $\textrm{PGL}_3(\textbf{L})$. Up to $\textbf{k}^*$ we can find a representant $A \in \textrm{GL}_3(\textbf{L})$ of this automorphism such that $A^g = A_g^{-1}AA_g, \,g \in \textrm{Gal}(\textbf{L},\textbf{k})$ (Lemma \ref{AutRep}) for a $A_g \in \textrm{GL}_3(\textbf{L})$ representing $\alpha_g$. The determinant of this representant is in $\textbf{k}^*$ and its class in $\textbf{k}^*/(\textbf{k}^*)^3$ only depends on $\alpha$ and not the representants $A_g$ of the $\alpha_g$, the isomorphism $\varphi$ or the splitting field (Lemma \ref{Determinant}). Even more, this map gives us a group homomorphism \[det: Aut_{\textbf{k}}(S) \rightarrow \textbf{k}^*/(\textbf{k}^*)^3\] which is the abelianization of the automorphism group of $S$ (Lemma \ref{AutAbel}). Now to use this with birational transformations we need to look at Mori fibre spaces. A rank $r$ fibration is a surjective morphism $X \rightarrow B$ where $X$ is a geometrically rational surface with relatively ample anticanonical divisor, $r \geq 1$, the morphism has connected fibers, the relative Picard rank equals $r$ and $B$ is a point or a smooth curve (Definition \ref{RankRFibration}). A Mori fibre space is a rank $1$ fibration. We look at the groupoid $BirMori(S)$ consisting of all birational maps between two Mori fibre spaces birational to $S$. Let now $X_i \rightarrow B_i, i=1,2$ be rank $i$ fibrations such that there is a birational morphism $X_2 \rightarrow X_1$ inducing a morphism $B_1 \rightarrow B_2$. We now say that $X_2 \rightarrow B_2$ dominates $X_1 \rightarrow B_1$. Now for the rank $2$ fibration fixed we get exactly two dominated Mori fibre spaces up to isomorphism. The induced map between them is called Sarkisov link (Definition \ref{SarkisovLink}). If we do the same thing with rank $2,3$ fibrations we get relations on those Sarkisov links called elementary relations. In \cite{Iskovskikh} and \cite{Lamy_2020} it is shown that $BirMori(S)$ is generated by Sarkisov links and isomorphisms where all relations are generated by the trivial and elementary relations. For Severi-Brauer surfaces, one can prove that the Sarkisov links between Mori fibre spaces birational to $S$ are links based at $3$ or $6$-points and are between $S$ and $S^{op}$. The Mori-fibre spaces birational to $S$ are also isomorphic to $S$ or $S^{op}$ (\cite{BSY}). In this paper we prove the following theorem, which will be the main part to find the abelianization of the group of birational transformations of a non-trivial Severi-Brauer surface.
\begin{maintheorem} \label{MainA}
    Let $S$ be a non-trivial Severi Brauer surface over $\textbf{k}$ and $S^{op}$ its opposite Severi-Brauer surface. Let $\chi, \chi_i: S \dashrightarrow S^{op}, \, \tau, \tau_i:S^{op} \dashrightarrow S$ be $\textbf{k}$-Sarkisov links based on $3$-points such that there are $\alpha_i, \delta_i \in Aut_{\textbf{k}}(S), \, \gamma_i, \beta_i \in Aut_{\textbf{k}}(S^{op})$ such that $\tau_i = \alpha_i \tau \beta_i, \, \chi_i = \gamma_i \chi \delta_i$ and the base points of $\chi, \tau$ are in general positions. Now if $\tau_3 \chi_3 \tau_2 \chi_2 \tau_1 \chi_1 = 1$ we have, that: \[\prod_{i=1}^3 det(\alpha_i\delta_i)det(\beta_i\gamma_i)^{-1}\]
\end{maintheorem}
\noindent We take the subgroupoid $G_S$ of $BirMori(S)$, in which all birational maps are between $S,S^{op}$. $G_S$ differs from $BirMori(S)$ only by some isomorphisms. We then choose a set $R$ of representants for every class of $3$,$6$-links, where the two links are equivalent iff they only differ by automorphisms. We then extend the determinant to $G_S$ in the following way: \[det_R: G_S \rightarrow \textbf{k}^*/(\textbf{k}^*)^3\] \[R \mapsto 1\] \[\alpha \in Aut_{\textbf{k}}(S) \mapsto det(\alpha)\] \[\beta \in Aut_{\textbf{k}}(S^{op}) \mapsto det(\beta)^{-1}\] To prove, that this is well defined we need to prove, that the trivial and elementary relations are sent onto $1$. Now in \cite{BSY} it is shown, that the elementary relations are all like in Theorem \ref{MainA} and are thus sent onto $1$ by the determinant. Thus the only relations left to consider are the trivial ones, which is done in Lemma \ref{DetTrivRef}. We can then reduce the determinant to a group homomorphism of $Bir_{\textbf{k}}(S)$ and make a direct sum with the group homomorphism from \cite{BSY}, this map is the ablianization of said group.
\begin{maintheorem} \label{MainB}
Let $q \in \mathcal{E}_3$. The abelianization of the group of birational transformations of a non-trivial Severi-Brauer surface $S$ over $\textbf{k}$ is given by: 
    \[\Phi:Bir_{\textbf{k}}(S) \rightarrow \bigoplus_{p \in \mathcal{E}_3 \setminus \{ q \}} \mathbb{Z}/3\mathbb{Z} \oplus \bigoplus_{p \in \mathcal{E}_6} \mathbb{Z} \oplus \textrm{DET}\]
    where $\textrm{DET} := det(Aut_{\textbf{k}}(S) \subseteq \textbf{k}^* / (\textbf{k}^*)^3$ and $\Phi$ is the direct sum of the determinant and a group homomorphism from Theorem A in \cite{BSY} which sends a Sarkisov link of class $p$ to $1_p$ for all $p \in \mathcal{E}_3 \setminus \{ q \} \cup \mathcal{E}_6$ and ignores the automorphisms and the Sarkisov links of class $q$.
\end{maintheorem}
\noindent This also implies, that if there are no $6$-points in $S$ the abelianization of $Bir_{\textbf{k}}(S)$ is a $3$-torsion group. This is the first example of a variety in which the abelianization of the group of birational transformations contains non-trivial classes of automorphisms. The surjectivity of the automorphism part remains open and can be looked at from an algebraic perspective via central simple algebras or a geometric perspective via different splitting fields of $3$-points and their norm map.
\section*{Acknowledgements}
\noindent I would like to thank my thesis supervisor Jérémy Blanc for introducing me to this problem and for the countless discussions. I would like to thank the authors of \cite{BSY} to let me continue the work they have already done on this problem, results proven by them in unpublished notes are marked as ($\star$). I would like to thank Gabriel Dill on giving me some examples for field extentions of degree $3$ and results on their norm map. 
\section{Preliminaries}
\noindent
\subsection{Mori Fibre Spaces and Sarkisov links}
\begin{definition} \label{RankRFibration}
    A surjective morphism $X \rightarrow B$ is called rank $r$ fibration for $r \geq 1$, if $X$ is a smooth surface, $-K_X$ is relatively ample, the fibres of $X \rightarrow B$ are connected, the relative picard rank equals $r$ and $B$ is either a point or a smooth curve. A rank $1$ fibration is also called Mori fibre space.
\end{definition}

\begin{definition}
    Let $ X \rightarrow B, X' \rightarrow B'$ be rank $r$ resp. $r'$-fibrations such that there is a birational morphism $X' \rightarrow X$ and there is a morphism $B \rightarrow B'$ such that:
    \[\xymatrix{X' \ar[r] \ar[d] & B' \\ X \ar[r] & B \ar[u] }\]
    We then say $X' \rightarrow B'$ dominates $X \rightarrow B$.
\end{definition}
\begin{definition} \label{RPiece}
    The piece of a rank $r$ fibration $X \rightarrow B$ is the $(r-1)$ dimensional combinatorial polytope constructed as follows: Each rank $d$ fibration dominated by $X \rightarrow B$ is a $(d-1)$-dimensional face. For each pair of faces $X_i \rightarrow B_i, i=1,2$ $X_1 \rightarrow B_1$ lies in $X_2 \rightarrow B_2$ if and only if the second fibration dominates the first one.
\end{definition}
\begin{definition} \label{SarkisovLink}
    Take $X_2 \rightarrow B_2$ a rank $2$ fibration. By the $2$-rays-game there are exactly $2$ Mori fibre spaces $X \rightarrow B, X' \rightarrow B'$ dominated by $X_2 \rightarrow B_2$. 
    \[ \xymatrix{ & X_2 \ar[dl] \ar[dr] &\\X \ar@{-->}[rr] \ar[d]& & X' \ar[d] \\ B \ar[dr]  & & B' \ar[dl] \\ & B_2 &}\]
    The induced map between $X,X'$ is called Sarkisov link. This is also the piece of the fibration $X_2 \rightarrow B_2$.
\end{definition}

\begin{theorem}[\cite{Lamy_2020}] \label{ElementaryRelation}
    A $2$-piece is homeomorphic to a disk and its boundary is a sequence of Sarkisov links, whose product is an automorphism. We say this product encodes an elementary relation between Sarkisov links.
\end{theorem}
\begin{theorem}[\cite{Iskovskikh} and \cite{Lamy_2020}] \label{Isk}
    Let $X$ be a smooth projective surface over a perfect field $\textbf{k}$, that is birational to a Mori fibre space. The groupoid 
    \[BirMori(X) := \{X_1 \dashrightarrow X_2 \text{ birat. } : X_1,X_2 \text{ Mori fibre spaces birat. to } X\} \]
    is generated by Sarkisov links and isomorphisms of Mori fibre spaces. Any relation between Sarkisov links is generated by trivial and elementary relations.
\end{theorem}
\noindent
The trivial relations are $\chi \circ \chi^{-1} = id$ and $\alpha \circ \chi \circ \beta = \chi'$ where $\chi, \chi'$ are Sarkisov links and $\alpha, \beta $ are automorphisms of Mori fibre spaces.
\subsection{Properties of Severi-Brauer surfaces}
For this we need the different possible Sarkisov-Links on Severi Brauer surfaces, which are worked out in \cite{BSY}. The most important $d$-points for those links are the $3$-points, which are the most basic ones, like shown in the next lemma. 
\begin{lemma}[{\cite[Corollary 2.2.2 with $p=3$]{BSY}}]
    Let $S$ be a non-trivial Severi-Brauer surface over a perfect field $\textbf{k}$. Then $S$ does not contain points of degree $d$, where $d$ is not divisible by $3$. On the other hand $S$ contains a point of degree $3$. 
\end{lemma}
\noindent
\begin{definition}
    Let $p$ be a $3$-point of $S$. Then the \textbf{splitting field of $p$} is the smallest field $\textbf{L} \supseteq \textbf{k}$ such that the components of $p$ ($p = \{ p_1, p_2, p_3 \}$) are defined over $\textbf{L}$.
\end{definition}
\noindent
We notice, that since $S_{\textbf{L}}$ has $\textbf{L}$-points we have $S_{\textbf{L}} \simeq \mathbb{P}^2_{\textbf{L}}$, this makes the splitting field an important tool to simplify the birational maps of Severi-Brauer Surfaces, especially because its Galois group is rather easy to handle in some cases.
\begin{lemma} [{\cite[Lemma 2.3.2]{BSY}}] \label{3pLem}
    Let $p$ be a $3$-point of $S$ and $\textbf{L}$ its splitting field. We find:
    \begin{enumerate}
        \item[(1)] The Galois group $Gal(\textbf{L}/\textbf{k})$ acts faithfully and transitively on $\{ p_1, p_2, p_3 \}$. In particular, it is either isomorphic to $\mathbb{Z}/3\mathbb{Z}$ or $Sym_3$ and contains a unique $g \in Gal(\textbf{L}/\textbf{k})$ of order $3$ such that $g(p_1) = p_2, g(p_2) = p_3, g(p_3) = p_1$.
        \item[(2)] Choosing this $g$ there is an $\textbf{L}$-isomorphism $\varphi : S_{\textbf{L}} \rightarrow \mathbb{P}^2_{\textbf{L}}$ such that $\varphi(p_1) = [1:0:0], \varphi(p_2) = [0:1:0], \varphi(p_3) = [0:0:1]$ and $\varphi \circ g \circ \varphi^{-1} = A_g \circ g$ for 
        $A_g = \big(\begin{smallmatrix}
        0 & 0 & \xi\\
        1 & 0 & 0 \\
        0 & 1 & 0
        \end{smallmatrix}\big)$
        for a $\xi \in \textbf{L}^g$.
        \item[(3)] If $q$ is a $3$-point of $S$ having the same splitting field as $p$, there is a $\alpha \in Aut_{\textbf{k}}(S)$ that sends $p$ to $q$. 
    \end{enumerate}
\end{lemma}
\noindent
We can notice, that $S \rightarrow \{ P\}$ is a Mori fibre space. Now to work with the birational maps we are interested in the Sarkisov links of $S$. Every Severi-Brauer surface corresponds to a central simple algebra $A$. We define $A^{op}$ the ring with the same set but reversed multiplication ($a \times_A b = b \times_{A^{op}} a$). This algebra corresponds to another Severi-Brauer surface $S^{op}$, which is not $\textbf{k}$-isomorphic to $S$ (see \cite{BSY} 3.2.5/3.2.6). $S^{op}$ is especially important, because all Sarkisov links go from $S$ to $S^{op}$. 
\begin{lemma}[3.2.3 and 3.2.7 from \cite{BSY}] \label{Links36}
    Any Sarkisov link $\chi: S \dashrightarrow S'$ is centered on a $3$-point or a $6$-point and its inverse is centered on a $3/6$-point of $S'$ with equal splitting field. We also have $S' \simeq S^{op}$ and that any Mori fibre space birational to $S$ is isomorphic to either $S$ or $S^{op}$
\end{lemma}
\begin{definition}
    Let $\chi : X_1 \dashrightarrow X_2$ and $\chi':X'_1 \dashrightarrow X'_2$ be two Sarkisov links between two rank $1$ del Pezzo surfaces over $\textbf{k}$. We say $\chi, \chi'$ (or respectively their base points) are equivalent if there are isomorphisms $\alpha : X_1 \rightarrow X'_1, \, \beta : X_2 \rightarrow X'_2$ such that $\beta \circ \chi = \chi' \circ \alpha$.
\end{definition}
\noindent
This equivalence classes are needed to work with the trivial relation $\alpha \circ \chi \circ \beta = \chi'$ and are a key component of the group homomorphism found in \cite{BSY}.
\begin{lemma} [{\cite[Lemma 3.3.2]{BSY}}]
    With $\chi, \chi'$ as above the following are equivalent:
    \begin{enumerate}
        \item[(1)] $\chi, \chi'$ are equivalent
        \item[(2)] There is an isomorphism like $\alpha$ from above that sends the base points of $\chi$ to the ones of $\chi'$
        \item[(3)] There is an isomorphism like $\beta$ from above that sends the base points of $\chi^{-1}$ to the ones of $\chi'^{-1}$
    \end{enumerate}
\end{lemma}
\begin{proof}
    \underline{(1) $\implies$ (2)}: Let $p,p'$ be the base points of $\chi,\chi'$. Since $\beta \circ \chi = \chi' \circ \alpha$ we get, that $p=Ind(\chi) = Ind(\chi' \circ \alpha) = \alpha^{-1}(Ind(\chi')) = \alpha^{-1}(p')$.
    \newline
    \underline{(2) $\implies$ (3)}: We define $\beta := \chi' \circ \alpha \circ \chi^{-1}$. We prove that this is an automorphism, then it will also act on the base points as described in (3). We write $\chi = \pi \circ \rho^{-1}, \, \chi' = \pi' \circ \rho'^{-1}$ the minimal resolutions. Then since $\alpha$ sends the points blown-up in $\rho$ to the ones blown-up in $\rho'$ we get, that $\rho'^{-1} \circ \alpha \circ \rho$ is a automorphism sending the points blown-up in $\pi$ to the ones blown-up in $\pi'$, proving that $\beta$ is an automorphism.
    \newline
    \underline{(3) $\implies$ (1)}: The same way as in the second implication we can prove, that $\alpha := \chi'^{-1} \circ \beta \circ \chi$ is an automorphism. This is enough to prove (1).
\end{proof}
\noindent
We define $\mathcal{E}_3, \mathcal{E}_6$ the equivalence class of $3/6$-points/links. The next lemma will give us the elementary relation between Sarkisov links.
\begin{lemma}[{\cite[Lemma 3.3.4]{BSY}}] \label{ElRel}
    The elementary relations are given by $\chi_6 \circ \chi_5 \circ \chi_4 \circ \chi_3 \circ \chi_2 \circ \chi_1 = id$, where for $i= 1,2,3$ we have $\chi_{2i -1}: S \dashrightarrow S^{op}$ are links from the same class and $\chi_{2i}: S^{op} \dashrightarrow S$ are links from the same class.
\end{lemma}
\noindent
We can see in the next image, that each link maps the base points of the neighbouring links to each other, this will help in constructing this relation later on. The relation is given by the following hexagon (from \cite{BSY}):
\newline
\begin{figure}[ht] 
	\includegraphics[width=7cm]{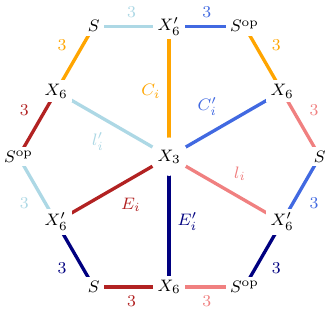}\centering
	\caption{Relation of Lemma~\ref{ElRel}. The centre is a del Pezzo surface $X_3$ of degree $3$ of Picard rank $3$. Each segment denotes the blow-up of a point of degree~$3$ (Figure from \protect\cite{BSY}). \label{fig: rel}}
\end{figure}

\begin{theorem} [{\cite[Theorem A]{BSY}}] \label{SBGHom}
    Let $q \in \mathcal{E}_3$ then there is a surjective group homomorphism:
     \[\Phi:Bir_{\textbf{k}}(S) \rightarrow \bigoplus_{p \in \mathcal{E}_3 \setminus \{ q \}} \mathbb{Z}/3\mathbb{Z} \oplus \bigoplus_{p \in \mathcal{E}_6} \mathbb{Z}\]
     which sends the links $\chi$ of equivalence class of $p \in \mathcal{E}_3 \setminus \{ q \} \cup \mathcal{E}_6$ to $1_p$ and the links of equivalence class $q$ to $0$.
\end{theorem}
\noindent
The reason why one of the links vanishes in the group homomorphism, is to make it surjective. In the proof for the abelianization we will see, why it is necessary to choose a $3$-link to vanish.

\section{The determinant map}
\noindent To find a useful group homomorphism that does not send the automorphisms to the neutral element we want to look at how automorphisms can be represented in the projective plane over a splitting field of $S$. In this chapter we are going to see some properties of the links and a good representation of the birational maps, to work with the elementary relation.
\begin{convention}
    To simplify the notation we always write $\textbf{k}$ for a perfect field and $S$ a non-trivial Severi-Brauer surface over $\textbf{k}$. 
\end{convention}
\noindent
For automorphisms, we find a good representation for an isomorphism to the projective plane over a field over which $S$ splits. This representation is key in defining the group homomorphism.
\begin{lemma}[$\star$] \label{AutRep}
    Let $\textbf{k} \subset \textbf{L}$ be a finite Galois extension such that $S_{\textbf{L}} \simeq \mathbb{P}^2_{\textbf{L}}$ via some isomorphism $\varphi$. Then there are $[A_g] \in \mathrm{PGL}_3(\textbf{L}), \forall g \in Gal(\textbf{L}/\textbf{k})$ such that $\varphi \circ g \circ \varphi^{-1} = [A_g] \circ g$. Then for all $\alpha \in Aut_{\textbf{k}}(S)$ we find a representant $A \in \mathrm{Gl}_3(\textbf{k})$ (unique up to $\textbf{k}^*$) of $\varphi \circ \alpha \circ \varphi^{-1}$ such that 
    \begin{equation} \label{eq:Rep}
        A^g = A_g^{-1}AA_g, \forall g \in Gal(\textbf{L}/\textbf{k})
    \end{equation}
\end{lemma}
\begin{proof}
    The $[A_g]$ exist, since $\varphi \circ g \circ \varphi^{-1} \circ g^{-1}$ is an automorphism and we just fix representants $A_g$ of their class in $\mathrm{PGl}_3(\textbf{L})$. We take a representant $A$ of  $\varphi \circ \alpha \circ \varphi^{-1}$ and find, since $\alpha$ is defined over $\textbf{k}$, that $[A_g] \circ g \circ [A] = \varphi \circ g \circ \alpha \circ \varphi^{-1} = \varphi \circ \alpha \circ g \circ \varphi^{-1} = [A] \circ [A_g] \circ g$ and thus $[A_g \circ A^g] = [A \circ A_g]$. Which implies, that there is $\lambda_g \in \textbf{L}$ such that $\lambda_g A^g = A_g^{-1} A A_g$. We calculate for $g,h \in Gal(\textbf{L}/ \textbf{k})$ 
    \[A^{gh} \lambda_{gh} = A_{gh}^{-1} A A_{gh} = (A_h^{-1})^g A_g^{-1} A A_g A_h^g = \lambda_g (A_h^{-1} A A_h)^g = \lambda_g g(\lambda_h) A^{gh}\]
    Thus $g \mapsto \lambda_g$ is a cocycle of the action of $Gal(\textbf{L}/\textbf{k}),\textbf{L}^*)$, but this is only consists of the class of $1$-cobords by \cite{serre2004corps}. Thus there is $\lambda \in \textbf{L}^*$ such that $\lambda_g = \frac{\lambda}{g(\lambda)}$, by replacing $A$ by $\lambda^{-1} A$ we are finished.
\end{proof}
\begin{definition}
    Let $\textbf{k} \subset \textbf{L}$ be a finite Galois extension such that $S_{\textbf{L}} \simeq \mathbb{P}^2_{\textbf{L}}$ via some isomorphism $\varphi$. Then there are $[A_g] \in \mathrm{PGL}_3(\textbf{L}), \forall g \in Gal(\textbf{L}/\textbf{k})$ such that $\varphi \circ g \circ \varphi^{-1} = A_g \circ g$. The representant $A$ of $\varphi \circ \alpha \circ \varphi^{-1}$ that satisfies \ref{eq:Rep} is called affine representant.  
\end{definition}
\begin{theorem}[$\star$] \label{Determinant}
    Let $\alpha \in Aut_{\textbf{k}}(S)$. For its representant $A$ that satisfies (\ref{eq:Rep}) we get $det(A) \in \textbf{k}^*/(\textbf{k}^*)^3$. Moreover $det(A)$ is independent of the field extension, the isomorphism and the $A_g$. This gives a group homomorphism from $Aut_{\textbf{k}}(S)$ to $\textbf{k}^*/(\textbf{k}^*)^3$. 
\end{theorem}
\begin{proof}
    Since $A^g = A_g^{-1} A A_g$ we find $g(det(A)) = det(A), \forall g \in Gal(\textbf{L}/\textbf{k})$. The $(\textbf{k}^*)^3$ is due to the uniqueness of $A$ being only up to $\textbf{k}^*$. For $\alpha, \beta$ with representation $A,B$ we find, that $(AB)^g = A^g B^g = A_g^{-1}A A_g A_g^{-1} B A_g = A_g^{-1} AB A_g$ thus $AB$ represents $\alpha \circ \beta$. Thus this determinant is a homomorphism. \\ Now for the uniqueness, we first notice that our choice of $A_g$ is only unique up to a scalar, which cancels out in $A_g^{-1} A A_g$. If we fix $\textbf{L}$, for two possible isomorphisms $\varphi, \varphi'$ we notice, that $\varphi' = \gamma \circ \varphi$ for a $\gamma \in Aut_{\textbf{L}}(\mathbb{P}^2)$. Using a representant $C$ of $\gamma$, we notice, we now can choose $C A_g (C^{-1})^g$ instead of $A_g$ and $CAC^{-1}$ instead of $A$. Thus we find, that $(CAC^{-1}) = C^g A^g (C^{-1})^g = C^g A_g^{-1} A A_g (C^{-1})^g = (C A_g (C^{-1})^g)^{-1} CAC^{-1} (CA_g (C^{-1})^g)$. Thus $CAC^{-1}$ is our representant, and this does not change the determinant. Now for different field extensions, we first show the independence for $\textbf{L} \subset \textbf{F}$ where we use $\varphi$ defined on $\textbf{L}$ and a extended version of $\varphi$ on $\textbf{F}$. We take $A, A_g \in \mathrm{GL}_3(\textbf{F}), \, g \in Gal(\textbf{F}/\textbf{k})$ such that they satisfy \ref{eq:Rep}. The elements in $Gal(\textbf{F}/\textbf{L})$ are elements from $Gal(\textbf{F}/\textbf{k})$ which are the identity when restricted to $\textbf{L}$. Thus for $g \in Gal(\textbf{F}/\textbf{L})$ we get, that $A_g = \varphi \circ (\varphi^{-1})^g = \varphi \circ \varphi^{-1} = id$. Thus $A_g \in \textbf{F}^*I_3$. Thus $A^g = A_g^{-1} A A_g = A$, which proves, that $A$ is defined over $\textbf{L}$. Now every element $g \in Gal(\textbf{L}/\textbf{k})$ is the restriction of an element from $g' \in Gal(\textbf{F}/\textbf{k})$ and thus $A_g = \varphi \circ (\varphi^{-1})^g = \varphi \circ (\varphi^{-1})^{g'} = \varphi \circ (\varphi^{-1})^g$. Thus $A^g = A^{g'} = A_g^{-1} A A_g$, when we choose the same $A_g = A_{g'}$. Thus over $\textbf{L}$ we can choose the same matrices and the determinant does not change. Now for two extensions $\textbf{k} \subset \textbf{L}, \textbf{k} \subset \textbf{L}'$ using isomorphisms $\varphi, \varphi'$ we take $\textbf{F}:= \textbf{L}\textbf{L}'$ and use the independence for $\textbf{L}, \textbf{L}\textbf{L}'$ and $\textbf{L}', \textbf{L}\textbf{L}'$ and one time the independence for the extensions of $\varphi, \varphi'$.  
\end{proof}
\noindent 
We now want to define a map $det: Bir_{\textbf{k}}(S) \rightarrow \textbf{k}^* / (\textbf{k}^*)^3$ using the determinant. For this, we choose a representant $\chi_p$ of every class $p \in \mathcal{E}_3 \cup \mathcal{E}_6$ of $\textbf{k}$-links and collect them in the set $R$. Then $R,Aut_{\textbf{k}}(S),Aut_{\textbf{k}}(S^{op})$ generate a subgroupoid of $BirMori(S)$, which we call $G_S$. We then want to define the following groupoid homomorphism: \[det_R: G_S \rightarrow \textbf{k}^*/(\textbf{k}^*)^3\] \[R \mapsto 1\] \[\alpha \in Aut_{\textbf{k}}(S) \mapsto det(\alpha)\] \[\beta \in Aut_{\textbf{k}}(S^{op}) \mapsto det(\beta)^{-1}\] A priori this is only defined on the words and not on $G_S$. To see that this is well defined on $G_S$, we need to check, that it sends the elementary and trivial relations to $1$. For the trivial relations we only need to look at $\chi^{-1} \beta \chi \alpha^{-1} = id$ or equivalently $\beta \chi = \chi \alpha$. To do this we need some further tools. 
\begin{lemma} \label{BaseP}
    Let $\chi: S \dashrightarrow S^{op}$ be a Sarkisov link based at $p$ and let $p^{op}$ be the base point of $\chi^{-1}$. Let $\textbf{L}$ the splitting field, and $\varphi, \varphi^{op}$ fitting isomorphisms between $S_{\textbf{L}}$ and $\mathbb{P}^2_\textbf{L}$. If $p,p^{op}$ are $3$-points we use for $\varphi, \varphi^{op}$ the isomorphisms from Lemma \ref{3pLem} and observe $\psi := \varphi^{op} \circ \chi \circ \varphi^{-1} = D\sigma$, where $D$ is a diagonal. If $p,p^{op}$ are $6$-points $\psi := \varphi^{op} \circ \chi \circ \varphi^{-1}$ is the blow-up of $6$ points and blow-down of the 5 conics through respectively $5$ of the points.
\end{lemma}
\begin{proof}
    We find that $\psi$ has $3$/$6$ proper base points. We take $A_g, A_g^{op} \in PGl_3(\textbf{k})$ such that $A_g \circ g = \varphi \circ g \circ \varphi^{-1}$ and similar for $A_g^{op}$ and get, since $\chi$ is $\textbf{k}$-birational that $A_g^{op} \circ g \circ \psi = \psi \circ A_g \circ g$ and thus $A_g^{op} \circ \psi^g = \psi \circ A_g$. We can also see, for a $g \in Gal(\textbf{L}/\textbf{k})$ and a element $C$ from the homaloidal system, since $(A_{g}^{-1 })^* \circ \psi^{g*} \circ A_g^{op*} = \psi^*$, that $C = A_g(C_2^g)$ for another element $C_2$ from the system ($A_g^{op*}$ preserves lines). Then for a $q_i$ the multiplicity of $\varphi(q_i)$ in $C_2$ is the same as of $\varphi(g(q_i)) = A_g(g(\varphi(q_i)))$ (which is one of the components of $\varphi(q)$) in $A_g(C_2^g) = C$. By choosing fitting $g$ we can show that all the base points of $\psi$ have equal multiplicity. Thus in the case of a $3$-point $3deg(\psi) - 3 = 3\alpha_0$ by Noethers relation and thus $deg(\psi) = \alpha_0 + 1$. Now by the second relation we get, that $\alpha_0(\alpha_0 + 2) = (\alpha_0 + 1)^2 - 1 = deg(\psi)^2 - 1 = 3 \alpha_0^2$ which gives us, that $\alpha_0 = 1$ and thus $deg(\psi) = 2$. Thus by our choice of $\varphi$ we get $\psi = \sigma$. Similarly in the $6$-point case we get the characteristic mentioned above with equal methods. Since all base points are proper we find for $\pi_{\psi},\pi_{\psi}'$ the minimal resolution of $\psi$ and the exceptional divisors $E'_i$ of $\pi'_{\psi}$, that $F_i:=\pi_{\psi}(E'_i)$ are curves of degree $2$. Now let $C$ be a general element, then $10 = C \cdot F_i = \tilde C \cdot \tilde F_i + 2\delta_i = 2\delta_i$ (irreducible conics in the plane are smooth, $C$ is general, thus $\tilde C \cdot E'_i = 0$), where $\delta_i$ is the number of base points contained in $F_i$. This gives us $\delta_i = 5$.
\end{proof}
\noindent
Thus in the $3$-point case $Exc(\chi)$ is a \textbf{k}-curve, with $3$ $\bar{\textbf{k}}$ components that correspond to lines through $2$ of the $3$ base points. Next we look at the image of $3$-points under Sarkisov links.
\begin{lemma} \label{ImageOf3P}
    Let $p,q$ be two different $3$-points of S and $\chi$ a Sarkisov link based at $q$. Then $\chi(p)$ is a $3$-point with the same splitting field as $p$.
\end{lemma}
\begin{proof}
    Let $g \in G:= Gal(\bar{\textbf{k}}/\textbf{k})$, then we find since $\chi$ is $\textbf{k}$-birational, that $g(\chi(p_i)) = \chi(g(p_i))$ this implies, that $\chi(p)$ is a $d$-point and thus, since $|\chi(p)| \leq 3$ we get $\chi(p)$ is a $3$-point. Let $\textbf{L}$ be the splitting field of $p$ and $\textbf{L}'$ be the one of $\chi(p)$. Since $p$ is defined over $\textbf{L}$ and $\chi$ is defined over $\textbf{k}$ we find $\chi(p)$ is defined over $\textbf{L}$, thus $\textbf{L}' \subseteq \textbf{L}$. \newline 
    If $p \not\in Exc(\chi)$ we can use $\chi^{-1}$ to get $\textbf{L} = \textbf{L}'$. \newline
    If $p \in Exc(\chi)$ we prove, that $p \sim q$. Using a isomorphism to $\mathbb{P}^2$, like the $\varphi$ in Lemma \ref{3pLem} we can see, that $Exc(\chi)$ consists of three curves between two of the respective points of $q$, and every curve is sent onto the point not on it, thus showing $p \sim q$ is actually enough. We have $q_i$ are defined over $\textbf{L}$, and now prove, that no $g \in Gal(\textbf{L}/\textbf{k}) \setminus \{ id \}$ fixes all $q_i$ at the same time. Then we find, that $q$ splits only over $\textbf{L}$ (and not over a strict subfield), which implies, that $\textbf{L}' = \textbf{L}$. Before this we notice that every $p_i$ is in exactly one of the curves in $Exc(\chi)$, elsewise $\chi(q)$ would not be a $3$-point. If there was such a $g$, we can see, that it does not fix all $p_i$. We can assume by reordering, that $g(p_1) = p_2$ and that $p_1 \in C_1 \subseteq Exc(\chi)$ and then $p_2 = g(p_1) \in g(C_1) = C_1$ ($g$ preserves $Exc(\chi)$ and fixes all $q_i$), which is not possible, since every $p_i$ is in exactly one of the three curves of $Exc(\chi)$. 
\end{proof}

\begin{lemma}
    Let $p$ be a $3$-point of $S$ and $q$ be a $3$-point of $S^{op}$ with the same splitting field $\textbf{L}$. We choose isomorphisms $\varphi:S_{\textbf{L}}\rightarrow \mathbb{P}^2_{\textbf{L}}$ and similar $\varphi^{op}$. Then we get for every $\psi: S \dashrightarrow S^{op}$ that $deg(\varphi^{op} \circ \psi \circ \varphi^{-1}) \equiv 2 ($mod $3)$ and for $\psi: S \dashrightarrow S$ that $deg(\varphi \circ \psi \circ \varphi^{-1}) \equiv 1 ($mod $3)$.
\end{lemma}
\begin{proof}
    It is of course true for automorphisms. We can write $\psi = \chi_k \circ \dots \circ \chi_1$  where $\chi_i$ are sarkisov links. We are going to prove this by induction:
    \newline
    \underline{k=1}: We find if $\psi = \chi_1$ we get, that $\varphi^{op} \circ \chi_1 \circ \varphi^{-1}$ is either of characteristic $2;1^3$ or $5;2^6$, by Lemma \ref{BaseP}. We find $2 \equiv 2 ($ mod $ 3)$ and $5 \equiv 2 ($ mod $ 3)$. We also find, that the points belonging to one $3/6$-point have the same multiplicity. \newline
    \underline{$k \rightarrow k+1$}: We define $\sigma_i := \varphi^{op} \circ \chi_i \circ \varphi^{-1}$ for $i$ odd and $\sigma_i := \varphi \circ\chi_i \circ (\varphi^{op})^{-1}$ for $i$ even. Then we need to look at $\sigma_{k+1} \circ \dots \circ \sigma_1$ by induction $deg(\sigma_{k+1} \circ \dots \circ \sigma_2) = 3r + \delta$, where $\delta = 2$ if $k$ is odd and $\delta = 1$ if $k$ is even. If $\sigma_1$ is of characteristic $2;1^3$ we write $\alpha$ for the multiplicity of the base points of $\sigma_1^{-1}$ in $\sigma_{k+1} \circ \dots \circ \sigma_2$. Then we get, that $deg(\sigma_{k+1} \circ \dots \circ \sigma_1) = 2(3r + \delta) - 3\alpha = 3(2r - \alpha) + 2 \delta$ which is equivalent to $1$ modulo $3$ if $\delta = 2$ and to $2$ modulo $3$ if $\delta = 1$. The multiplicity of the base points of $\sigma_1^{-1}$ are now $3r + \delta - 2 \alpha$, so they are equal. The multiplicity of the other points are still the same. Now if the characteristic of $\sigma_1$ is $5;2^6$. Then $\sigma_1$ is the blow-up of $6$ points and the blow down of the $6$ conics through $5$ of the points. Let $\alpha$ be the multiplicity of the base points of $\sigma_1^{-1}$ in $\sigma_{k+1} \circ \dots \circ \sigma_2$, then we find the intersection of the strict transform of a element $C$ of the homaloidal map of $\sigma_{k+1} \circ \dots \circ \sigma_2$ with the strict transform of the conics is $2(3r + \delta) - 5\alpha$, this is then the new multiplicity of those points. The other points keep their multiplicity. Using Noethers first equality we find, that $deg(\sigma_{k+1} \circ \dots \circ \sigma_1) = 5(3r + \delta) - 12 \alpha = 3(5r - 4 \alpha) + 5\delta$ which is equivalent to $1$ if $\delta =2$ and equivalent to $2$ if $\delta = 1$.
\end{proof}
\noindent
For the next Theorem we will use $3$-points with a splitting field that is a degree $3$ extension. Thus we make the following convention.
\begin{convention}
    For a splitting field $\textbf{L}$ of a $3$-point $p := \{ p_1, p_2, p_3\}$ of $S$ we always take $g$ to be the element of the Galois group such that $g(p_1) = p_2, \, g(p_2) = p_3, \, g(p_3) = p_1$. We write $\varphi$ for an $\textbf{l}$-isomorphism between $S_{\textbf{L}}$ and $\mathbb{P}^2_\textbf{L}$ with the properties given in Lemma \ref{3pLem}. We also write $A_g$ for the matrices defined in \ref{3pLem}. 
\end{convention}

\begin{proposition} \label{AffineRep}
    Let $p$ be a $3$-point of $S$ whose splitting field $\textbf{L}$ satisfies $[\textbf{L}:\textbf{k}]=3$. We then take the conventional $\varphi, \varphi^{op}$ for $p$ and a $3$-point $q$ of $S^{op}$ with the same splitting field. We find, that $Gal({\textbf{L}/\textbf{k}})$ is generated by our $g$. Now let $\psi: S_{\textbf{k}} \dashrightarrow S^{op}_{\textbf{k}}$ and we can find $F_0,F_1,F_2 \in k[x,y,z]$ homogenous such that $\varphi^{op} \circ \psi \circ \varphi^{-1} = [F_0:F_1:F_2]$, $d$ such that its degree is $3d + 2$ and:
    \begin{equation} \label{eq:RepB}\xi^{i(d_1+1)} (F_0^{g^i}, F_1^{g^i}, F_2^{g^i}) = (A_{g^i}^{op})^{-1} \circ (F_0,F_1,F_2) \circ A_{g^i}, \, i=0,1,2 \end{equation}
    Now let $\psi: S_{\textbf{k}} \dashrightarrow S_{\textbf{k}}$, we can find $G_0,G_1,G_2 \in k[x,y,z]$ homogenous such that $\varphi \circ \psi \circ \varphi^{-1} = [G_0:G_1:G_2]$, $d_1$ such that its degree is $3d_1 + 1$ and:
    \begin{equation}\label{eq:RepB2} \xi^{i\cdot d_2} (G_0^{g^i}, G_1^{g^i}, G_2^{g^i}) = A_{g^i}^{-1} \circ (G_0,G_1,G_2) \circ A_{g^i}, \, i=0,1,2 \end{equation}
\end{proposition}
\begin{proof}
    Due to Lemma \ref{3pLem} we can take $A_g = 
    \begin{pmatrix} 
        0 & 0 & \xi \\
        1 & 0 & 0 \\
        0 & 1 & 0 
    \end{pmatrix}$, we then also take $A_{id} = I_3$, $A_{g^2} = A_g^2$, then $A_g^3 = \xi I_3$ and similar for $A_g^{op}$ just with $\xi^{-1}$ instead of $\xi$. We will just do the first case, the second will work similarly. We take $F := \varphi^{op} \circ \psi \circ \varphi^{-1}$. Since $\psi$ is $\textbf{k}$ birational we get, that $F \circ A_{g^i} \circ g^i = A_{g^i}^{op} \circ g^i \circ F$ which gives us, that $F^{g^i} = g^{i} \circ F \circ g^{-i} = (A_{g^{i}}^{op})^{-1} \circ F \circ A_{g^{i}}$. Thus there is $\mu_{g^{i}} \in \textbf{L}^*$ such that
    \[(A_{g^{i}}^{op})^{-1} \circ (F_0, F_1, F_2) \circ A_{g^{i}} = (F_0^{g^{i}},F_1^{g^{i}},F_2^{g^{i}})\mu_{g^{i}}\]
    Where $\mu_{id} = 1$. We then find for $u_1,u_2$ in the Galois group, that there is $\lambda_{u_1,u_2} \in \textbf{L}^*$ such that $A_{u_1 \circ u_2} = \lambda_{u_1,u_2} A_{u_1} u_1(A_{u_2})$ and similar $\lambda_{u_1,u_2}^{op}$. Since $Gal(\textbf{L}/\textbf{k}) = \{id,g,g^2 \}$ we get due to our choice of $A_{g^i}$, that all $\lambda_{u_1,u_2} = 1$ except for $\lambda_{g^2,g^2} = \lambda_{g^2,g} = \lambda_{g,g^2} = \xi^{-1}$ and $\lambda_{g^2,g^2}^{op} = \lambda_{g^2,g}^{op} =\lambda_{g,g^2}^{op} = \xi$. We then find:
    \[\mu_{u_1u_2}(F_0^{u_1u_2},F_1^{u_1u_2},F_2^{u_1u_2}) = (\lambda_{u_1,u_2}^{op} A_{u_1}^{op} u_1(A_{u_2}^{op}))^{-1} \circ (F_0,F_1,F_2) \circ (\lambda_{u_1,u_2} A_{u_1} u_1(A_{u_2})) \] 
    \[= \frac{\lambda_{u_1,u_2}^{deg(F)} \mu_{u_1}}{\lambda_{u_1,u_2}^{op}} u_1((A_{u_2}^{op})^{-1} \circ (F_0^{u_1}, F_1^{u_1}, F_2^{u_1}) \circ A_{u_2}) = \frac{\lambda_{u_1,u_2}^{deg(F)} \mu_{u_1} u_1(\mu_{u_2})}{\lambda_{u_1,u_2}^{op}} (F_0^{u1u2},F_1^{u1u2},F_2^{u1u2}) \]
    Thus $\frac{\lambda_{u_1,u_2}^{3d + 2} \mu_{u_1} u_1(\mu_{u_2})}{\lambda_{u_1,u_2}^{op}} = \mu_{u_1,u_2}$ and moreover the fraction $\frac{\lambda_{u_1,u_2}^{3d + 2}}{\lambda_{u_1,u_2}^{op}}$ is equal to $1$ except in the case where $(u_1,u_2) \in \{ (g,g^2), (g^2,g),(g^2,g^2) \}$ in which it is equal to $\xi^{-3(d+1)}$ We then get the equations:
    \[\frac{\mu_{g^2}}{\xi^{2(d+1)}} = \frac{\mu_g}{\xi^{d+1}} \frac{g(\mu_g)}{\xi^{d+1}}, \, \mu_{id} = \frac{\mu_g}{\xi^{(d+1)}} \frac{g(\mu_{g^2})}{\xi^{2(d+1)}} = \frac{\mu_{g^2}}{\xi^{2(d+1)}} \frac{g^2(\mu_g)}{\xi^{d+1}}, \, \frac{\mu_g}{\xi^{d+1}} = \frac{\mu_{g^2}}{\xi^{2(d+1)}}\frac{g^2(\mu_{g^2})}{\xi^{2(d+1)}}\]
    We then observe, that $g^{i} \mapsto \frac{\mu_{g^i}}{\xi^{i(d+1)}}$ is a 1-cocycle. And since by \cite{serre2004corps} $H^1(Gal(\textbf{L}/\textbf{k}),\textbf{L}^*) = \{ 1 \}$ there is $\epsilon \in \textbf{L}^*$ such that $\frac{\mu_{g^{i}}}{ \xi^{i(d+1)}} = \frac{\epsilon}{g^{i}(\epsilon)}$ replacing $F_i$ with $\epsilon^{-1} F_i$ gives the result.
\end{proof}

\begin{definition}
    Let $p$ be a $3$-point of $S$ such that for its splitting field $\textbf{L}$ we have $[\textbf{L}:\textbf{k}]=3$. We then take the conventional $\varphi, \varphi^{op}$ for $p$ and a $3$-point $q$ of $S^{op}$ with the same splitting field. We find, that $Gal({\textbf{L}/\textbf{k}})$ is generated by our $g$. Now let $\psi: S_{\textbf{k}} \dashrightarrow S^{op}_{\textbf{k}}$. For the $F_i \in k[x,y,z]$ found in Lemma \ref{AffineRep} the map $\mathbb{A}^3 \rightarrow \mathbb{A}^3, x \mapsto (F_0(x),F_1(x),F_2(x))$ is called affine representant.
\end{definition}
\noindent
\begin{lemma} \label{matrixL}
    The affine representants of automorphisms of $S$ correspond to the following matricies: 
    \[\begin{pmatrix} 
    a & \xi g(c) & \xi g^2(b) \\
    b & g(a) & \xi g^2(c) \\
    c & g(b) & g^2(a)
\end{pmatrix}\]
where $(a,b,c) \in \textbf{L}^3 \setminus \{ (0,0,0) \}$.
\end{lemma}
\begin{proof}
Let $\alpha \in Aut_{\textbf{k}}(S)$. Theorem \ref{AffineRep} gives the same result as Lemma \ref{AutRep}, thus that the affine representant $A$ of $\alpha$ satisfies $A^g = A_g^{-1} A A_g$. Let $a_i$ be the columns of $A$. Then the classes of $a_i$ are the image of the coordinate points $[1:0:0],[0:1:0],[0:0:1]$ which form a 3-point. We can observe that $A^g = A_g^{-1} A A_g \Leftrightarrow A_g(g(a1)) = a_2, A_g(g(a_2)) = a_3, A_g(g(a_3)) = \xi a_1$. Thus our matrix is given by:
\begin{equation}
\begin{pmatrix} 
    a & \xi g(c) & \xi g^2(b) \\
    b & g(a) & \xi g^2(c) \\
    c & g(b) & g^2(a)
\end{pmatrix} \end{equation}
Where $a_1 = (a,b,c) \neq (0,0,0)$.
\end{proof}
\noindent Every $\textbf{L}$-birational map of $\mathbb{P}^2$ that has such affine representants corresponds via $\varphi$ to a $\textbf{k}$-birational map of $S$, since it does not change under the Galois action. We also notice, that such affine representants always come with $\textbf{L},\varphi, \varphi^{op}, A_g$. Since this only works in splitting fields, that are degree $3$ extensions, we show in the next lemma the existence of such a field. 
\begin{lemma}
    There is a $3$-point of $S$ whose splitting field $\textbf{L}$ is a degree $3$ extension of $\textbf{k}$.
\end{lemma}
\begin{proof}
    With Lemma 2.3.4 in \cite{BSY} we find a $\textbf{L}$ such that the field extension of $\textbf{k} \subseteq \textbf{L}$ is of degree 3 and $S$ has points over $\textbf{L}$ (it is isomorphic to $\mathbb{P}^2$ over $\textbf{L}$). We want to prove, that $\textbf{L}$ is a splitting field of a $3$-point. We take the pre-image of $[1:0:0]$ with respect to the isomorphism. Then we take the Galois orbit of this point with respect to $Gal(\textbf{L}/\textbf{k})$. This is then a $d$-point for $d \leq 3$. Thus it must be a $3$-point, since there are no points or $2$-points or $1$-points in $S$. We then get, that the splitting field is in $\textbf{L}$ and is not $\textbf{k}$ thus it must be $\textbf{L}$, since $3$ is prime. 
\end{proof}
\noindent 
A useful tool, to find good splitting fields over which an automorphism is diagonal, is the fixpoint.

\begin{lemma}[$\star$] \label{Fixpoints}
    Let $\alpha \in Aut_{\textbf{k}}(S) \setminus \{ id \}$. Then there is exactly one $3$-point in $S$ whose components are fixed by $\alpha$ and its components are the only $\bar{\textbf{k}}$-points, fixed by $\alpha$.
\end{lemma}
\begin{proof}
We take the degree $3$ splitting field $\textbf{L}$ of a point $p$ from the last lemma and the according $\varphi$ from Lemma $\ref{3pLem}$. We now take $A$ to be an affine representant of $\alpha$ and see, that ist Eigenvectors correspond exactly to the fixpoints of $\alpha$. Those do not necessaryly need to be defined over $\textbf{L}$, but since the Eigenvalues come from a degree $3$ polynomial we can take $\textbf{L}'$ a Galois extention of $\textbf{k}$ that lies over $\textbf{L}$ such that they are defined. We expand $\varphi$ to $\textbf{L}'$ and replace $A$ by the (possibly with a different scalar due to the new galois group) affine representant of $\alpha$ over $\textbf{L}'$. Now all Eigenvalues are defined. Thus take $v, \lambda$ an Eigenvector and its Eigenvalue. We get, that for all $u \in Gal(\textbf{L}'/ \textbf{k})$ $A^u = A_u^{-1} A A_u$ and thus $A_u u(v)$ is an Eigenvector of $A$ with Eigenvalue $v$. Take $V_1, \dots , V_r$ the Eigenspaces of $A$. And take $V$ the vectorspace generated by all of them. Thus $A_u u(V) = V$. Thus we get, that the preimage of the projectivization of $V$ under $\varphi$ is a twisted linear subvariety and thus $V$ has dimension $3$. But $V = V_1$ is only possible for the identity. If we have $r=2$ and w.l.o.g. $dim(V_1) = 2, \, dim(V_2) = 1$ we get, that for $v, w$ a base of $V_1$ with Eigenvalues $\lambda_1$ we see, that $A_u u(v), A_u u(w)$ are Eigenvectors of Eigenvalue $u(\lambda)$, but $dim(V_2) = 1$ thus we get, they are in $V_1$ and thus $A_uu(V_1)) = V_1$ and thus the preimage of the projectivization is a twisted linear curve, which is impossible, thus $r=3$ and $dim(V_i) = 1$. But then we find, that the preimages of the projectivization of $V_i$ are $3$ $\textbf{L}'$-points, which are as a set Galois invariant, thus they must form a $3$-point.
\end{proof}
\noindent
Next we can finaly prove that the determinant of the trivial relation is trivial.

\begin{lemma}[$\star$] \label{DetTrivRef}
    If for $\alpha \in Aut_{\textbf{k}}(S), \,  \beta \in Aut_{\textbf{k}}(S^{op})$ and $\chi$ a $\textbf{k}$-link of $S$ we have $\beta \chi = \chi \alpha$ then $det(\beta) = det(\alpha)^{-1}$.
\end{lemma}
\begin{proof}
    If we choose a different representative $\chi' = \gamma_1 \chi \gamma_2$ for $\textbf{k}$-automorphisms $\gamma_1, \gamma_2$ of $S^{op},S$, we can replace $\beta, \alpha$ by $\gamma_1^{-1} \beta \gamma_1, \gamma_2 \alpha \gamma_2^{-1}$. Since their determinant did not change, it is enough to show the lemma for one representant. We take $p,p'$ the base points of $\chi, \chi^{-1}$. Thus we know, that $p$ is fixed as a set by $\alpha$ and $p'$ by $\beta$. Take $p_i,p_i'$ the components of $p,p'$ over $\bar{\textbf{k}}$.  \newline
    We first assume, that $p,p'$ are $3$-points. We now take $\textbf{L}$-isomorphisms $\varphi, \varphi^{op}$ of $S,S^{op}$ to $\mathbb{P}^2$ like in Lemma \ref{3pLem}. Then we can see, that $\varphi^{op} \chi \varphi^{-1}$ has $\{ [1:0:0], [0:1:0], [0:0:1]\}$ as base points and its inverse too. Thus it is equal to $D \sigma$ for some diagonal $D$ (if we choose a nice order on $p'$). Since $\Sigma:= (yz,xz,xy)$ is an affine representant $\sigma$ correponds to a $\textbf{k}$-birational map from $S$ to $S^{op}$. We can thus change our representant and choose $D=I_3$. Now since $p,p'$ are fixed as $3$-points by $\alpha, \beta$, the affine representants $A,B$ of $\alpha, \beta$ are diagonals multiplied by permutation matrices. Those permutation matrices cannot be transpositions. Assume otherwise $\alpha$ would permute $p_1,p_2$ and fix $p_3$. Then if $g \in Gal(\textbf{L}/\textbf{k})$ that corresponds to a $3$-cycle and sends $p_3$ to $p_1$, we get, that $g(p_3) = g(\alpha(p_3)) = \alpha(g(p_3)) = \alpha(p_1) = p_2)$, which is a contradiction. In the case were the splitting field has degree $6$ it can also not correspond to a $3$-cycle with the same type of argument. We see that in the case of a degree $3$ splitting field $A_g, A_g^{op}$ from Lemma \ref{3pLem} satisfy the condition to be representants of $\textbf{k}$-automorphisms if we choose the same matrices as the $A_g$ used to define the affine represenatation ($A_g^g = A_g = A_g^{-1}A_g A_g$). Thus we can take $A = D_A A_g^{i_A}, \, B = D_B (A_g^{op})^{i_B}$ where $0 \leq i_A,i_B \leq 2$ and $D_A,D_B$ are diagonal. We find that there is $\lambda \in \textbf{L}$ such that $\lambda B  \Sigma =  \Sigma A$, but since both $ B  \Sigma , \, \Sigma A$ are affine representants we have $\lambda \in \textbf{k}$, thus by changing $B$ we can assume $\lambda = 1$. Then $D_B (A_g^{op})^{i_B} \Sigma = B  \Sigma =  \Sigma A = det(D_A) \xi^{i_A} D_A^{-1} (A_g^{op})^{i_A} \Sigma$ (in affine coordinates). Thus we find $det(B) = det(D_A)^2 \xi^{2i_A} = det(A)^2 = det(A)^{-1}$ in $\textbf{k}^*/(\textbf{k}^*)^3$. \newline
    Now if $p,p'$ are $6$-points, we can take $q$ the fixpoint of $\alpha$ (componentwise fixpoint) from Lemma \ref{Fixpoints} and observe, that $q' := \chi(q)$ is the componentwise fixpoint of $\beta$. We now take $\textbf{L}$-isomorphisms $\varphi, \varphi^{op}$ of $S,S^{op}$ to $\mathbb{P}^2$ like in Lemma \ref{3pLem} for $q,q'$. We again write $A,B$ for the affine representants of $\alpha, \beta$ and we write $Q$ for the one of $\chi$. We can see, that $Q$ fixes the points $[1:0:0], [0:1:0],[0:0:1]$. since $q' = \chi(q)$. We have seen in Lemma $\ref{BaseP}$, that $Q$ is of degree $5$. Since it fixes $[1:0:0], [0:1:0],[0:0:1]$ but can not have them as base points its first entry has a $x^5$ component but no $y^5,z^5$ components and similar for the second and third entry. Since $[1:0:0], [0:1:0],[0:0:1]$ are fixed by $A,B$ we find $A,B$ are diagonal. We get another $\lambda \in \textbf{L}$ such that $\lambda B Q = Q A$ and can remove it like in the first case. Now we get, that $Q = B^{-1}Q A = B^{-1}Q(a_{11}x, a_{22}y, a_{33}z)$. Due to the $x^5$ term in the first entry we find $b_{11} = a_{11}^5$ and similarly $b_{22} = a_{22}^5, b_{33} = a_{33}^5$. Thus $B = A^5$ and we find, that $det(B) = det(A)^5 = det(A)^{-1}$ modulo $(\textbf{k}^*)^3$.
\end{proof}
\noindent
To work with the elementary relation requires more work and will thus be done in the next $3$ subchapters.

\section{Rewriting elementary relations}
 \noindent Take now an elementary relation, thus we take Sarkisov links $\chi_i: S \dashrightarrow S^{op}, \, \tau_i: S^{op} \dashrightarrow S, \, i=1,2,3$ based at $3$-points such that $\chi_1,\chi_2,\chi_3$ are equivalent and $\tau_1,\tau_2,\tau_3$ are equivalent and \[\tau_3\chi_3\tau_2\chi_2\tau_1\chi_1 = id\] Due to the equivalence we can find $\alpha_i, \delta_i \in Aut_{\textbf{k}}(S), \, \gamma_i, \beta_i \in Aut_{\textbf{k}}(S^{op})$ such that $\tau_i = \alpha_i \tau \beta_i, \, \chi_i = \gamma_i \chi \delta_i$ where $\chi, \tau$ are the links representing the respective class of links. In this section we prove that the choice of links representing the class does not affect the determinant of the word. In Section \ref{Sec5} we will then prove, that the determinant is always trivial in $\textbf{k}^*/(\textbf{k}^*)^3$.

\begin{lemma}\label{Subst}
    Take a elementary relation $\tau_3\chi_3\tau_2\chi_2\tau_1\chi_1 = id$ of a Severi-Brauer surface $S$. Choose $\tau: S^{op} \dashrightarrow S, \, \chi: S \dashrightarrow S^{op}$ such that $\chi_i \sim \chi, \tau_i \sim \tau$. Due to the equivalence we can find $\alpha_i, \delta_i \in Aut_{\textbf{k}}(S), \, \gamma_i, \beta_i \in Aut_{\textbf{k}}(S^{op})$ such that $\tau_i = \alpha_i \tau \beta_i, \, \chi_i = \gamma_i \chi \delta_i$. \\
    Then $\prod_{i=1}^3 det(\alpha_i\delta_i)det(\beta_i\gamma_i)^{-1}$ does not depend on the choice of $\chi, \tau$. \\
    In particular the determinant of the word does not depend on the choice of representators.
\end{lemma}
\begin{proof}
    We take $ \tau': S^{op} \dashrightarrow S, \, \chi': S \dashrightarrow S^{op}$ such that $\tau \sim \tau', \, \chi \sim \chi'$. Again, using the equivalence, we find $\alpha, \delta \in Aut_{\textbf{k}}(S), \, \gamma, \beta \in Aut_{\textbf{k}}(S^{op})$ such that $\tau' = \alpha \tau \beta, \, \chi' = \gamma \chi \delta$ and find $\tau_i = \alpha_i \alpha^{-1} \tau' \beta^{-1}\beta_i, \, \chi_i = \gamma_i \gamma^{-1}\chi' \delta^{-1}\delta_i$. The determinant of the word with the $\tau, \chi$ is $\prod_{i=1}^3 det(\alpha_i\delta_i)det(\beta_i\gamma_i)^{-1}$ and the one with $\tau', \chi'$ is $\prod_{i=1}^3 det(\alpha_i \alpha^{-1}\delta_i\delta^{-1})det(\beta_i\beta^{-1}\gamma_i\gamma^{-1})^{-1} = (\prod_{i=1}^3 det(\alpha_i\delta_i)det(\beta_i\gamma_i)^{-1}) (\frac{det(\beta\gamma}{\alpha\beta})^3 = (\prod_{i=1}^3 det(\alpha_i\delta_i)det(\beta_i\gamma_i)^{-1})$ in $\textbf{k}^*/(\textbf{k}^*)^3$.
\end{proof}
\begin{remark} \label{RSubst}
    The lemma holds even if $\chi \sim \tau, \, \chi \neq \tau$, which of course can not happen if both $\chi, \tau$ are representators of their class of $3$-links, but is is still usefull for the following lemma.
\end{remark}

\noindent
Now we want to prove an even stronger statement, where we take $\chi_1$ to represent its class of links and find, that fixing $\chi_1$ the determinant of the word only depends on the base points of $\tau_1$.

\begin{lemma} \label{DepPoint}
     Take two elementary relations:
     \[\tau_3\chi_3\tau_2\chi_2\tau_1\chi = id\]
     \[\tau_3'\chi_3'\tau_2'\chi_2'\tau_1'\chi = id\]
     of a Severi-Brauer surface $S$ where $Ind(\tau_1) = Ind(\tau_1')$. Then the determinant of the words are equal in $\textbf{k}^*/(\textbf{k}^*)^3$. 
\end{lemma}
\begin{proof}
    Using Lemma \ref{Subst} we can choose $\chi, \tau := \tau_1$ to represent their classes of links (even if $\chi, \tau$ are equivalent) and find $\alpha_i, \delta_i \in Aut_{\textbf{k}}(S), \, \gamma_i, \beta_i \in Aut_{\textbf{k}}(S^{op}) $ such that $\tau_i = \alpha_i \tau \beta_i, \, \chi_i = \gamma_i \chi \delta_i$ and $\alpha_i', \delta_i' \in Aut_{\textbf{k}}(S), \, \gamma_i', \beta_i' \in Aut_{\textbf{k}}(S^{op})$ such that $\tau_i' = \alpha_i' \tau \beta_i', \, \chi_i' = \gamma_i' \chi \delta_i'$. We take $p,p',q,q'$ the base points of $\chi, \chi^{-1},\tau,\tau^{-1}$. Since $Ind(\tau_1) = Ind(\tau_1')$ and by our choice of representors we have, that $\alpha_1 = \delta_1 =  \delta_1' =id, \beta_1=\gamma_1=\beta_1'= \gamma_1' = id$. Due to the structure of the elementary relation we can find, that both $\delta_2, \delta_2'\alpha_1'$ send $\tau_1(p')$ to $p$. Thus there is $d_1 \in Aut_{\textbf{k}}(S)$ that fixes $p$ such that $d_1\delta_2 = \delta_2'\alpha_1'$. We can then also find with Lemma \ref{DetTrivRef} a $d_1' \in Aut_{\textbf{k}}(S^{op})$ such that $d_1' \chi = \chi d_1$ and $det(d_1')det(d_1) = 1$. We can thus rewrite the second word to: 
    \[\tau_3'\chi_3'\tau_2'\chi_2'\tau_1'\chi = \tau_3'\chi_3'\tau_2'\gamma_2'\chi\delta_2'\alpha_1'\tau_1\chi\] \[= \tau_3'\chi_3'\tau_2'\gamma_2'\chi d_1\delta_2\tau_1\chi = \tau_3'\chi_3'\tau_2'\gamma_2'd_1'\chi \delta_2\tau_1\chi\]
    Thus we can repeat this process and use this to find similar $d_i,d_i', i = 2, \dots ,5$ such that $\beta'_2\gamma'_2 d'_1 = d_2 \beta_2 \gamma_2 , \, \delta_3'\alpha_{2}'d'_{2} = d_{3} \delta_3\alpha_2, \, \beta_3'\gamma_3' d'_{3} = d_{4} \beta_3 \gamma_3, \, \alpha_3'd'_{4} = d_5 \alpha_3$ and \[id = \tau_3'\chi_3'\tau_2'\chi_2'\tau_1'\chi = d_5\tau_3\chi_3\tau_2\chi_2\tau_1\chi = d_5\]
    Thus since $det(d_i)det(d_i')=1$ we can calculate, that the determinants are the same.
\end{proof}
\noindent We notice that $id = \tau_3\chi_3\tau_2\chi_2\tau_1\chi_1 = \alpha^{-1}\tau_3\chi_3\tau_2\chi_2\tau_1\chi_1 \alpha, \, \alpha \in Aut_{\textbf{k}}(S)$. Thus since given the base points of $\chi_1^{-1}$, $\chi_1$ is unique up to right multiplication with an automorphism we can go even further and find, that the determinant of the word $\tau_3 \chi_3 \tau_2\chi_2 \tau_1 \chi_1$ only depends up to $(\textbf{k}^*)^3$ on the base points of $\chi_1^{-1}$ and $\tau_1$ which are both in $S^{op}$.

\section{The determinant of an elementary relation} \label{Sec5}
\noindent Now if we have fixed base points of $\tau_1$ and fixed $\chi_1$ it is enough to choose one such possible relation and prove that the determinant of the corresponding word is trivial. We also choose $\chi:= \chi_1$ to represent its class of links. By Lemma \ref{DepPoint} this will imply the result for all other relations with the same base points of $\tau_1$ and the same $\chi_1$. Thus we will now find such a relation and calculate the determinants. The next lemmas will tell us how to find given $\tau$ and $\alpha_1 \in Aut_{\textbf{k}}(S)$ suitable affine representants $A_2, \dots , A_6$ of automorphisms $\alpha_2, \dots , \alpha_6$ such that $\alpha_6 \tau \alpha_5 \chi \dots \tau \alpha_1 \chi = id$ is an elementary relation. We then take $\tau_i := \alpha_{2i} \tau \alpha_{2i-1}, \, \chi_i:= \chi$ and this will be our example. We start with the case where for the splitting field $\textbf{L}$ of the base point $p$ of $\chi$ we have $[\textbf{L}:\textbf{k}] = 3$. We choose the conventional isomorphism $\varphi$ for $p$ and $\varphi^{op}$ for $p'$ the base point of $\chi^{-1}$. We observe, that $\Sigma := (yz,xz,xy)$ is an affine representant of a $3$-link equivalent to $\chi$. The affine representant of $\chi$ only differs from $\Sigma$ by a diagonal matrix, which is itself an affine representant of an automorphism, since $\Sigma$ is one. Thus using Lemma \ref{Subst} we can assume, that $\Sigma$ is an affine representant of $\chi$. 
\begin{convention}
    For this subchapter we take $\chi:S \dashrightarrow S^{op}$ a $3$-link with base point $3$ whose splitting field $\textbf{L}$ satisfies $[\textbf{L}:\textbf{k}] = 3$. We also take the $\varphi$ for $p$ and the $\varphi^{op}$ for the base point of $\chi^{-1}$. We recall, that we have $\xi$ such that $A_g = \begin{pmatrix}
        0 & 0 & \xi \\
        1 & 0 & 0 \\
        0 & 1 & 0
    \end{pmatrix}$ and $A_g^{op} = \begin{pmatrix}
        0 & 0 & \xi^{-1} \\
        1 & 0 & 0 \\
        0 & 1 & 0
    \end{pmatrix}$ where $A_g = \varphi \circ (\varphi^{-1})^g, \, A_g^{op} = \varphi^{op} \circ ((\varphi^{op})^{-1})^g$.
\end{convention}
\subsection{The relation with only one class of points}
We start with the case where $\tau = \chi$.
\begin{lemma}\label{ElemRel}
    Let $\Sigma_i, \, i=1,2,3$ be affine representants of equivalent $3$-links $\tau_i$. If the projectivization $\sigma_i$ of $\Sigma_i$ maps $\{[1:0:0],[0:1:0],[0:0:1] \}$ to itself and $\sigma$ maps the base points of $\sigma_{i}^{-1}$ to the base points of $\sigma_{i+1}, \, i=1,2$ and $\tau_3 \chi \tau_2 \chi \tau_1 \chi = id$, then the last expression is an elementary relation. 
\end{lemma}
\begin{proof}
    We compare this to the hexagon in Figure \ref{fig: rel}. 
    The $X_6, X_6'$ exist since the $3$-links are a composition of blow-ups of the $3$-points, which are their base points. Among the six del Pezzo surfaces of degree $6$ we only obtain two $X_6, X_6'$, as respectively three links are equivalent, which implies there are automorphism between their base points. Now we notice that for one link the preimage of the blow-up of the base points of the previous and also the preimage of the base points of the next link are equal by the action of the links on those points. Thus we can blow-up $X_6$ or respectively $X_6'$ in those points and get to a degree $3$ del Pezzo. Then the blow-up morphisms always blow-up the same points up to reordering and we can thus choose the blow-ups in a way such that they always start at the same del Pezzo $X_3$.
\end{proof}
\begin{lemma} \label{FixMatr}
  Take $\rho:S^{op} \dashrightarrow S$ a $3$-link. Let $p,p'$ be the base points of $\chi, \chi^{-1}$. Let $A$ be a affine representant of a $\alpha \in Aut_{\textbf{k}}(S^{op})$ and $P$ an affine representant of $\rho$, whose existence follows from Proposition \ref{AffineRep}. \\ We then find, that $A^{\rho}D_{\xi}$ is a affine representant of an automorphism of $S$, where $A^{\rho}$ is the matrix which has columns given by $P$ applied to the columns of $A$ and $D_{\xi}$ is the diagonal matrix with diagonal $1, \xi, \xi^2$. Its inverse $A' := D_{\xi}^{-1} (A^{\rho})^{-1}$ is an affine representant of an $\alpha' \in Aut_{\textbf{k}}(S)$, such that $\alpha' \rho \alpha$ sends $p'$ to $p$.  
\end{lemma}
\begin{proof}
      Let $a_i, \, i=1,2,3$ be the columns of $A$ and $b_i := P(a_i)$. We then get, since \[A_g^{op}(g(a_1) \mid g(a_2) \mid g(a_3)) = A_g^{op} A^g = A A_g^{op} = (a_2 \mid a_3 \mid \xi^{-1} a_1)\] that $a_2 = A_g^{op}(g(a_1)), \, a_3 = A_g^{op} (g(a_2)) = (A_g^{op})^2(g^2(a_1)), \, a_1 = \xi A_g^{op}(g(a_3)) = \xi (A_g^{op})^3 (a_1) = a_1$. Since we have $A_g  P^g (A_g^{op})^{-1} = \xi P$ we get, that $\xi b_2 = A_g(g(b_1)), \, \xi b_3 = A_g(g(b_2)) = \xi^{-1} A_g(g(\xi b_2)),  b_1 = \xi A_g(g(b_3)) = \xi^{-1} A_g^3(b_1) = b_1$. Thus if we look at $A^{\rho} D_{\xi} = (b_1 \mid \xi b_2 \mid \xi^2 b_3)$ we get that $A^{\rho} D_{\xi}$ has the fitting representation and thus corresponds to an element in $Aut_{\textbf{k}}(S)$. We know $\rho A$ sends $\{ [1:0:0], [0:1:0],[0:0:1]\}$ to the classes of $P(a_i)$ which are send back to $\{ [1:0:0], [0:1:0],[0:0:1]\}$ by $A'$. All other matrices with this property send the classes of $P(a_i)$ to $\{ [1:0:0], [0:1:0],[0:0:1]\}$. Since using $\varphi, \varphi^{op}$ with $p,p'$, $\{ [1:0:0], [0:1:0],[0:0:1]\}$ corresponds to $p,p'$.      
\end{proof}
\begin{remark}
    In Lemma \ref{FixMatr}, each other automorphism $\alpha''$ which is chosen such that $\alpha' \rho \alpha$ sends $p$ to $p'$ can be found by multiplying an automorphism whose affine representant is a diagonal to $\alpha'$.
\end{remark}
\noindent
 Going forward we can apply the same argument with $A_3, \dots , A_6$ and find all possible relations with this method, by multiplying diagonals to $D_{\xi}^{-1}((A_i)^{\sigma})^{-1}$ (except for $A_6$ where one cannot choose the diagonal freely anymore). Using the results of the last subchapter we will and can choose the examples where the diagonals are $I_3$ to make the calculation simpler. We can now look at an example where we can see, how this calculation works, and that in this case the product of the determinants becomes $1$.
\begin{example}
    We take an example from \cite{BSY}, where $\textbf{k} := \mathbb{C}(u,v)$ and $\textbf{L} := \textbf{k}(\sqrt[3]{u})$. One can find, that this is a Galois extention of degree $3$ and that its norm map is non surjective, not having $\xi := v^{-1}$ in its image. Thus using \cite{BSY} Lemma 2.3.3 there is a Severi-Brauer surface $S_{\xi}$ and an isomorphism $\varphi: S_{\xi} \rightarrow \mathbb{P}_{\textbf{L}}^2 $ such that for $g \in Gal(\textbf{L}/ \textbf{k})$ a generator, we have $\varphi \circ g \circ \varphi^{-1} = A_g \circ g$, where $A_g$ is as usual. Now we want to find an example for the relation and check wether the determinant behaves, like we want it to. For this we first try to find an $\alpha_1 \in Aut_{\textbf{k}}(S_{\xi})$. For this we need an invertible matrix $A_1$ such that $A_gA_1^g = A_1A_g$. For simplicity we choose $A_1$ to be over $\textbf{k}$, thus we need $A_gA_1 = A_1A_g$. To do this we start with a column (or a first point) and use this equations to find the other column (its "Galois orbit"). One example for this would be 
    $A_1 := \begin{pmatrix}
        1 & 3v & 2v \\
        2 & 1 & 3v \\
        3 & 2 & 1 \\
    \end{pmatrix}$.
    Now we want to find fitting $A_i, \, i=2, \dots , 6$ such that $A_6 \sigma A_5 \sigma \dots A_1 \sigma = id$. But we have more information in our relation. we know, that $A_i \sigma A_{i-1}$ always sends $\{ [1:0:0], [0:1:0], [0:0:1] \}$ to itself. We assume, that there is no permutation within the set and thus we can find for $a_{i}$ being the columns of $A_1$, that we can choose $(\sigma(a_{1})  \sigma(a_{2})  \sigma(a_{3})  )$ as $A_2$ (where we apply an affine version of $\sigma$ to the $a_i$). We call this matrix $A_1^{\sigma}$. We can now check since $\sigma A_g = \xi A_g^{op} \circ \sigma$, that if we instead choose $A_2 := D_{\xi^{-1}} A_1^{\sigma}$ ,where $D_{\xi^{-1}}$ is the diagonal matrix with $1,\xi^{-1},\xi^{-2}$ on its diagonal, $A_2$ satisfies our restriction $A_g^{op} A_2 = A_2 A_g^{op}$. We can thus carry on with this principle and multiply with some fitting elements in $\textbf{k}$ we find the matrices:
     \[A_2 := \begin{pmatrix}
        36v^2 - 6 v & -3v & -18v + 4 \\
        -18 v^2 + 4v & 36v^2 - 6 v & -3v \\
        -3v^2 & -18 v^2 + 4v & 36v^2 - 6 v \\
    \end{pmatrix} \]
    \[A_3 := \begin{pmatrix}
        -9 v^2 + 2v & 6v^2 - v & 54v^3 - 21 v^2 + 2v \\
        54v^2 - 21 v + 2 & -9 v^2 + 2v & 6v^2 - v \\
        6v - 1 & 54v^2 - 21 v + 2 & -9 v^2 + 2v \\
    \end{pmatrix}\]
        \newline
     \[A_4 := \begin{pmatrix}
        1 & 2 & 3 \\
        3v & 1 & 2 \\
        2t-2 & 3v & 1 \\
    \end{pmatrix}, \, 
    A_3 := \begin{pmatrix}
        36v^2 - 6 v & -18 v^2 + 4v & -3t-2^2 \\
        -3t-2 & 36v^2 - 6 v & -18 v^2 + 4v \\
        -18 v + 4 & -3v & 36v^2 - 6 v \\
    \end{pmatrix}\] 
    \[A_6 := \begin{pmatrix}
        -9 v^2 + 2v & 54v^2 - 21 v + 2 &  6v - 1 \\
        6v^2 - v & -9 v^2 + 2v & 54v^2 - 21 v + 2 \\
        54v^3 - 21 v^2 + 2v & 6v^2 - v & -9 v^2 + 2v \\
    \end{pmatrix}\]
    and get, that $det(A_1A_3A_5)det(A_2A_4A_6)^{-1} = 1$. Now we still do not now what $A_6 \sigma A_5 \sigma \dots A_1 \sigma$ is, we just know its a diagonal matrix. But one can calculate this with a computer program and find, that it is the identity. Thus we have seen an example, where the two points have the same class and where our relation of the determinants becomes trivial. 
    
\end{example}
\noindent 
The next lemma calculates the matrices $A_2, \dots A_6$ dependant on $A_1$ similar to how it is done in the example. We recall the notion $A^P$ from Lemma \ref{FixMatr} and in this case specifically $A^{\Sigma}$ the matrix where you apply $\Sigma$ to the columns of $A$. For the case where the two classes of points in our relation are the same, this is almost enough. In the other case, we have some problems with our representation and knowing whether something is a $\textbf{k}$-automorphism. Those problems will be fixed in the next subchapter.
\begin{lemma} \label{LargeCalc}
    Let $\textbf{C}$ be a field and $A \in \mathrm{Gl}_3(C)$ ( where the classes of its columns in $\mathbb{P}^2_C$ together with the coordinate points are in general positions). Then if we define $M_1(A):= A, \, M_i(A) := (M_{i-1}(A)^{\Sigma})^{-1}, \, i=2, \dots , 6$, where $\Sigma = (yz,xz,xy)$, we get, that \[d(A):= det(A^{-1}M_3(A)^{-1}M_5(A)^{-1})det(M_2(A)M_4(A)M_6(A))\] \[= \frac{P(A)^{14}}{P(A^{-1})^7 det((A^{\sigma}))^{42}det(A)^{42}}\] where $P$ is the product of all matrix entries. And as birational maps over $\mathbb{P}^2$ we get, that $M_6(A) \sigma \dots M_2(A) \sigma A \sigma = D(A)$ where $D(A)$ is a diagonal matrix such that:
    \newline
    \[D(A)_{ii} = det(A)^6\frac{\prod_{k=1}^3(A^{-1})_{ik}^2}{\prod_{l=1}^3a_{li}}\]
    Thus $det(D(A)) =\frac{P(A^{-1})^2det(A)^{18}}{P(A)}$.
\end{lemma}
\begin{proof}
    To calculate the determinant we need to calculate $det(M_i(A)), i=2, \dots , 6$. We calculate directly:
    \[det(M_3(A))=-(P(A)det(A)det(M_2(A))^4)^{-1}\]
    With this we can calculate $det(M_i(A)), i \geq 3$ by using that $M_i(A) = M_{i-1}(M_2(A))$ and replacing $det(M_{i-1}(\cdot))$ with the previous equality
    \[det(M_4(A))=-P(M_2(A))^{-1}P(A)^4det(A)^4det(M_2(A))^{15}\]
    \[det(M_5(A))=P(M_3(A))^{-1}P(M_2(A))^4P(A)^{-15}det(A)^{-15}det(M_2(A))^{-56}\]
    \[det(M_6(A))=P(M_4(A))^{-1}P(M_3(A))^4P(M_2(A))^{-15}P(A)^{56}det(A)^{56}det(M_2(A))^{209}\]
    Thus we need to find $P(M_i(A)), i=2,3,4$:
    \[P(M_2(A)) = -P(A^{-1})P(A)^2det(A)^9det(M_2(A))^9\]
    Using that $P(M_2(A)^{-1}) = P(A)^2$ we find using again that $M_i(A) = M_{i-1}(M_2(A))$:
    \[P(M_3(A)) = P(A)^{-3}P(A^{-1})^2det(A)^9det(M_2(A))^{-9}\]
    \[P(M_4(A))=P(A^{-1})^{-3}P(A)^7det(A)^{-18}det(M_2(A))^{18}\]
    Putting this together we get:
    \[det(A^{-1}M_3(A)^{-1}M_5(A)^{-1})det(M_2(A)M_4(A)M_6(A)) = \frac{P(A)^{14}}{P(A^{-1})^7 det((A^{\sigma}))^{42}det(A)^{42}}\]
    The fact that $D(A)$ is an diagonal automorphism, is due to our choice of $M_2(A), \dots , M_6(A)$, which makes sure that the degree is $1$ and $[1:0:0], [0:1:0],[0:0:1]$ are fixed. This follows from the fact that $M_6(A) \sigma \dots \sigma M_1(A)$ is of characteristic $2;1^3$ and it and its inverse have $\{ [1:0:0], [0:1:0],[0:0:1]\}$ as their base points. We then calculate, that that the coordinate lines are contracted onto the base points not on them, which implies that $D(A)$ is diagonal. We also know with a similar argument as above just using the inverse relation, that this diagonal must map the classes of the columns of $(A^{-1})^{\sigma}$ to the classes of the columns of $M_6(A)$. Thus we now calculate the $M_i(A)$ and confirm the announced matrix $D(A)$. We observe, that for $\lambda_i(A) \in \textbf{k}^*, \, i = 2, \dots , 6$:
    \[M_2(A)_{ij} = (\prod_{k \neq i} a_{jk}) (A^{-1})_{ij} \lambda_2(A)\]
    \[\lambda_2(A) = \frac{det(A)}{det(A)^{\Sigma}}\]
    Using $M_i(A) = M_{i-1}(M_2(A))$ and dividing by $P(A), P(A^{-1})$ if needed, we get:
    \[M_3(A)_{ij} = \frac{\prod_{k \neq i}(A^{-1})_{jk}}{\prod_{k = 1}^3a_{ik}} \lambda_3(A)\]
    \[\lambda_3(A) = \lambda_2(M_2(A)) \lambda_2(A)^2P(A)\]
    \[M_4(A)_{ij} = \frac{1}{(\prod_{k \neq i}a_{jk})(\prod_{k=1}^3 (A^{-1})_{ik})} \lambda_4(A)\]
    \[\lambda_4(A) = \frac{\lambda_3(M_2(A))}{\lambda_2(A)}\]
    \[M_5(A)_{ij} = \frac{\prod_{k=1}^3a_{ik}}{\prod_{k \neq i}(A^{-1})_{jk} \prod_{k \neq i} \prod_{l \neq j} a_{kl}} \lambda_5(A)\]
    \[\lambda_5(A) = \frac{\lambda_4(M_2(A))}{P(A) \lambda_2(A)^2}\]
    \[M_6(A)_{ij} = \frac{(\prod_{k=1}^3(A^{-1})_{ik}) (\prod_{k \neq i}a_{jk}^2)a_{ji}}{(\prod_{l \neq j}a_{li}) (\prod_{k \neq i} \prod_{l \neq j}(A^{-1})_{kl})} \lambda_6(A)\]
    \[\lambda_6(A) = \frac{\lambda_5(M_2(A)) \lambda_2(A)}{P(A)}\]
    Now we also find:
    \[D(A)_{ii} = det(A)^6\frac{\prod_{k=1}^3(A^{-1})_{ik}^2}{\prod_{l=1}^3a_{li}}\]
    \[(A^{-1})^{\sigma}_{ij} = \prod_{l \neq i} (A^{-1})_{lj}\]
    Thus:
    \[(D(A)(A^{-1})^{\sigma})_{ij} = det(A)^6\frac{\prod_{k=1}^3(A^{-1})_{ik}^2\prod_{l \neq i} (A^{-1})_{lj}}{\prod_{l=1}^3a_{li}}\]
    and
    \[M_6(A){ij} = \frac{(\prod_{k=1}^3(A^{-1})^2_{ik}) (\prod_{k = 1}^3a_{jk}^2)(\prod_{l \neq i}(A^{-1})_{lj})}{(\prod_{l = 1}^3a_{li})} \frac{\lambda_6(A)}{P(A^{-1})}\]
    Thus for the diagonal matrix $D'(A)$ with $D'(A)_{jj} = (\prod_{j=1}^3a_{jk}^2) \frac{\lambda_6(A)}{P(A)det(A)^6}$ we get:
    \[M_6(A) = D(A) (A^{-1})^{\sigma}D'(A)\]
    Thus the diagonal sends the projective class of the columns of $(A^{-1})^{\sigma}$ to the ones of the columns of $M_6(A)$.
\end{proof}
\begin{remark} \label{RemDiag}
    Like already mentioned the matrix $M_2(A)$ is up to a diagonal the unique matrix $M$ such that $M \sigma A$ fixes the coordinate points. Now we can study what happens if we insert those diagonals in the calculation. Let $D_i, i=2, \dots ,6$ be diagonal matrices and define $M'_1(A) := A, \, M'_i(A) := D_i M_2(M'_{i-1}(A)), \, i = 2, \dots , 6$. Then we find, by using that for any diagonal $D$ $det(D)^{-1}M_2(DM_i(A)) = M_2(M_i(A))D, \, M_2(M_i(A)D) = D^{-2}M_2(M_i(A))$: 
    \[M'_3(A) = D_3 M_3(A) D_2, \, M'_4(A) = D_4 D_2^{-2} M_4(A)D_3, \] \[ M'_5(A) = D_5D_3^{-2}M_5(A)D_4D_2^{-2}, \, M'_6(A) = D_6 D_4^{-2}D_2^4 M_6(A)D_5D_3^{-2}\]
    And using $D \sigma D = \sigma$ this gives us:
    \[M'_6(A) \sigma \dots \sigma M'_1(A)\sigma = D_6 D_4^{-2}D_2^4 M_6(A) \sigma \dotsm \sigma A \sigma = D_6 D_4^{-2}D_2^4 D(A)\]
    One can also calculate that $det(M'_1(A)^{-1}M'_3(A)^{-1}M'_5(A)^{-1})det(M'_2(A)M'_4(A)M'_6(A)) = d(A) det(D_6)$, thus the insertion of diagonals does the same change as a multiplication on the left-hand side of the relation with a diagonal.
\end{remark}
\begin{lemma} \label{TrivEx}
    Let $A_1$ be an affine representant of a $\textbf{k}$-automorphism of $S^{op}$ (where we use the splitting field of $p$), where the classes of $A_1$'s columns in $\mathbb{P}^2_{\textbf{L}}$  together with the coordinate points are in general positions. We define with the notation of Lemma \ref{LargeCalc}: 
    \[A_i := D_{\xi^{t(i)}}M_2(A_{i-1}), \, i=2, \dots , 5 , \, A_6 := D(A_1)^{-1}D_{\xi}^3D_{\xi^{t(6)}}M_2(A_5)\]
    Where $t(i) = -1, \, i \text{ even }, \, t(i) = 1 , \, i \text{ odd}$. Then $A_i$ are affine representants of automorphisms $\alpha_i$ and $\alpha_6 \chi \dots \alpha_1 \chi = id$ and this is an elementary relation whose determinant is in $(\textbf{k}^*)^3$.
\end{lemma}
\begin{proof}
    We use Remark \ref{RemDiag} with $D_i = D_{\xi^{t(i)}}, i = 2, \dots , 5$ and $D_6 = D(A_1)^{-1}D_{\xi}^3D_{\xi^{t(6)}}$. Thus $A_6 \sigma \dots A_1 \sigma = D_6 D_4^{-2}D_2^4 = id$. \\
    The fact, that $A_i$ are affine representants comes from Lemma \ref{FixMatr} which implies, that $D_{\xi^{t(i)}}M_2(A_i)$ is an affine representant and we can calculate, that $D(A_1)^{-1}D_{\xi}^3$ is one too. In Lemma \ref{LargeCalc}, we can see that $d(A_1)$ equals $det(D(A_1))$ up to $(\textbf{k}^*)^3$. Thus we have using Lemma \ref{LargeCalc} that $det(A_2A_4A_6)det(A_1^{-1}A_3^{-1}A_5^{-1}) = det(D_{\xi})^6d(A_1) det(D(A_1))^{-1} \in (\textbf{k}^*)^3$, since by Lemma \ref{matrixL} we get, that $P(A_1) \in \textbf{k}^*$. The fact that this relation is elementary comes from the action of $A_i$ on the base points and Lemma \ref{ElemRel}. 
\end{proof}
\begin{theorem} \label{TrivialCase}
    All elementary relations of $S$ where both classes of base points are equivalent and their splitting field is a degree $3$ extension of $\textbf{k}$ have a trivial determinant.  
\end{theorem}
\begin{proof}
     We fix an affine representant $A_1$ (where the classes of $A_1$'s columns in $\mathbb{P}^2_{\textbf{L}}$ are in general positions) of a $\textbf{k}$-automorphism of $S$. Due to Corollary \ref{DepPoint} we just need to prove the theorem for one possibility for the other $5$ automorphism. In this case, we also have only one kind of link for which we will use the one represented by $\Sigma$. We choose the $A_i$ like it is given in Lemma \ref{TrivEx}, which implies that the determinant is trivial.
\end{proof}
\subsection{The relation with two classes of $3$-points}
Now if the base points of $\tau,\chi$ are not equivalent we want to again reduce this problem to the case of Lemma \ref{LargeCalc}. The problem now is that $\tau \neq \chi$ and we thus need some information on the base points of $\tau$ and its inverse. The next few results will give these points, an affine representant of $\tau$ and the $\textbf{k}$-automorphisms to find a relation.
\begin{lemma}
    Let $p,q \in S$ be two non-equivalent $3$-points with splitting fields $\textbf{L}_p, \textbf{L}_q$ such that $ [\textbf{L}_p:\textbf{k}] = [\textbf{L}_q:\textbf{k}]=3$. Let $\textbf{F}$ be the smallest common superfield of $\textbf{k}$ such that both points are defined. Then we find $Aut_{\textbf{k}}(\textbf{F}) \simeq Aut_{\textbf{k}}(\textbf{L}_p) \times Aut_{\textbf{k}}(\textbf{L}_q)$ 
\end{lemma}
\begin{proof}
    We take $\omega_p, \omega_q \in \Bar{\textbf{k}}$ such that $\textbf{L}_p = \textbf{k}(\omega_p), \textbf{L}_q = \textbf{k}(\omega_q)$ by the Primitive Element Theorem. Then we find, that $\textbf{F} = \textbf{k}(\omega_p,\omega_q)$. Since $p,q$ are non equivalent and both field extentions are of degree $3$, we find that $\textbf{L}_p \cap \textbf{L}_q = \textbf{k}$ and since $\textbf{k} \subset \textbf{L}_p, \textbf{k} \subset \textbf{L}_q$ are normal $\textbf{k} \subset \textbf{F}$ must be of degree $9$, because no zero of $min(\textbf{k}, \omega_q)$ can be in $\textbf{L}_p$ since it is normal over $\textbf{k}$. Then the map
    \[ Aut_{\textbf{k}}(\textbf{F}) \rightarrow Aut_{\textbf{k}}(\textbf{L}_p) \times Aut_{\textbf{k}}(\textbf{L}_q)\]
    \[g \mapsto (g_{|\textbf{L}_p}, g_{|\textbf{L}_q})\]
    is injective, since $\omega_p, \omega_q$ generate $\textbf{F}$ and thus since both sides contain $9$ elements it is also surjective and thus is an isomorphism.
    \newline
\end{proof}
\begin{convention} \label{con:2SF}
    From now on we write $\textbf{L}_p, \textbf{L}_q$ for the splitting fields of non-equivalent $3$-points $p,q$ of $S$ and $F := \textbf{L}_p\textbf{L}_q$. We always will identify $Aut_{\textbf{k}}(\textbf{F} )$ as $ Aut_{\textbf{k}}(\textbf{L}_p) \times Aut_{\textbf{k}}(\textbf{L}_q)$. We will also extend our isomorphisms $\varphi$ like in Lemma \ref{3pLem} (for $p$) to $\textbf{F}$. Since this can still be represented over $\textbf{L}_p$, the $\textbf{L}_q$ part of the Galois group of $\textbf{F}$ does not affect $\varphi$. In the case where $\textbf{L}_q$ is a degree $3$ extension we take $g'$ a generator of $\textrm{Gal}(\textbf{L}_q,\textbf{k})$. 
\end{convention}
\begin{lemma}
    Let $\textbf{k}, \, \textbf{L}_p, \, \textbf{L}_q, \, \textbf{F}, \varphi$ be like in Convention \ref{con:2SF}. For $A_g$ as usual we find, that for all $(g,u) \in Aut_{\textbf{k}}(\textbf{F}) \simeq Aut_{\textbf{k}}(\textbf{L}_p) \times Aut_{\textbf{k}}(\textbf{L}_q)$ we get, that $\varphi \circ (g,u) \circ \varphi^{-1} = A_g \circ (g,u)$ and $\varphi \circ (id,u) \circ \varphi^{-1} = (id,u)$. We then find that for each $3$-point $q'$ equivalent to $q$, $\varphi(q')$ is a fixpoint of $A_g \circ (g,id)$ and $\{ id \} \times Aut_{\textbf{k}}(\textbf{L}_q)$ permutes the components of $\varphi(q')$, like it does with $q'$ in $S_{\textbf{F}}$. The properties of Theorem \ref{AffineRep} also remain valid for the extension $\textbf{k} \subset \textbf{F}$, where we replace $g^{i}$ with $(g^{i},u)$ and choose $A_{(g^i,u)} = A_{g^i}$.  
\end{lemma}
\begin{proof}
    We take $\varphi$ as we did before for $\textbf{L}_p$, since it can be defined over $\textbf{L}$ we get, that $\varphi \circ (g,u) \circ \varphi^{-1} \circ (g,u)^{-1} = \varphi \circ (\varphi^{-1})^{(g,u)} = \varphi \circ (\varphi^{-1})^g = A_g$. And similarly for $(id,u)$. We can also find, that $\varphi \circ (g,id) \circ \varphi^{-1}$ fixes $q_i, i=1,2,3$ and thus the same holds for $A_g \circ (g,id)$ with $\varphi(q_i)$. Since $(id,u)$ permutes the $q_i$ and $\varphi \circ (id,u) \circ \varphi^{-1} = (id,u)$, the same holds for $(id,u)$ with $\varphi(q_i)$. Since the elements in $\varphi Bir_{\textbf{k}}(S) \varphi^{-1}$ can be defined over $\textbf{L}_p$ we can choose the same affine representation as in Theorem \ref{AffineRep}. 
\end{proof}
\noindent
To represent our $q$ in a good way we look at the next lemma. For this we need that $\textbf{k} \subset \textbf{L}_q$ is of degree $3$.
\begin{lemma} \label{3-pointNonEq}
    If for $3$-points $p, q$ of $S$, $\textbf{k} \subset \textbf{L}_q, \textbf{L}_p$ are of degree $3$ we can construct the following representations of $q$ over $\textbf{F} = \textbf{L}_p \textbf{L}_q$. There is a $b \in \textbf{F}$ such that $N_{\textbf{F} / \textbf{L}_q}(b) = \xi^{-2}$ and for each such $b$ we define  
    \[B(b) := \begin{pmatrix}
        1 & 1 & 1 \\
        \frac{1}{\xi g(b)} & \frac{1}{\xi g(g'(b))} & \frac{1}{\xi g(g'(g'(b)))} \\
        b & g'(b) & g'(g'(b))
    \end{pmatrix}\] The classes of $B(b)$'s columns form a $3$-point over $S$ via the usual isomorphism $\varphi$ (extended like in Convention \ref{con:2SF}) and each $3$-point equivalent to $q$ not on the coordinate lines can be given like this. We also have $A_gB(b)^g = \xi B(b) D_b^g$ where $D_b$ is a diagonal matrix with diagonal elements $b, g'(b),g'(g'(b))$. We also can see directly by definition, that $B(b)^{g'} = B(b) \begin{pmatrix}
        0 & 0 & 1 \\
        1 & 0 & 0 \\
        0 & 1 & 0
    \end{pmatrix}$ (where we will later refer to the right hand side matrix as $(123)$). 
\end{lemma}
\begin{proof}
    We look at a $[1:b_1:b_2]$ that satisfies $[1:\xi^{-1}g(b_2^{-1}): \xi^{-1}g(\frac{b_1}{b_2})] = [\xi g(b_2):1:g(b_1)] \stackrel{!}{=} A_g(g([1:b_1:b_2])) = [1:b_1:b_2]$ from which we can deduce $b_1 = g(b_2^{-1})\xi^{-1}$ and $N_{\textbf{F}/\textbf{L}_q}(b_2) = \xi^{-2} $. The fact that such a $3$-point is a $g'$ orbit gives us the structure of $B(b)$. Since $q$ exists and is not on one of the lines between the components of $p$, we get the existence of such $b, B(b)$. Conversely, we can see that for every $b \in \textbf{F}$ such that $N_{\textbf{F} / \textbf{L}_q}(b) = \xi^{-2}$ we get $A_g(g([1:\frac{1}{\xi g(b)}:b])) = [\xi g(b):1:\frac{1}{\xi g(g(b))}] = [1:\frac{1}{\xi g(b)}: b]$, thus the columns of $B(b)$ are fixed by $A_g g$. Since the columns of $B$ form an orbit with respect to $Aut_{\textbf{k}}(\textbf{L}_q)$ we get with the previous property, that the columns form a $3$-point. The properties $A_gB(b)^g = \xi B(b) D_b^g$ can be calculated using $A_g(g([1:\frac{1}{\xi g(b)}:b])) = [\xi g(b):1:\frac{1}{\xi g(g(b))}] = [1:\frac{1}{\xi g(b)}: b]$. By the construction of $B(b)$ we can directly calculate that \[B(b)^{g'} = B(b) \begin{pmatrix}
        0 & 0 & 1 \\
        1 & 0 & 0 \\
        0 & 1 & 0
    \end{pmatrix}\].
\end{proof}
\noindent
Given such matrices we can find a $\tau$ that has $q$ as base points if the splitting field of $q$ is a degree $3$ extension.

\begin{lemma} \label{BB'}
    Take $\textbf{L}_p, \textbf{L}_q, \textbf{F}, \varphi$ as in Convention \ref{con:2SF}. We can take $b \in \textbf{F}$ be such that $N_{\textbf{F} / \textbf{L}_q}(b) = \xi^{2}$ and with Lemma \ref{3-pointNonEq} we can define $B := B(b)$ (representing a $3$-point of $S^{op}$) and $B':=B(b^{-1})$ (representing a $3$-point of $S$). \\ There is a $\tilde \lambda \in L_p$ such that $\tilde \Sigma := \tilde \lambda B' \Sigma B^{-1}$ is an affine representation of a $3$-link from $S^{op}$ to $S$. We will later refer to the projectivisation of $\tilde \Sigma$ as $\tilde \sigma$.
\end{lemma}
\begin{proof}
The existence of $b$ follows from Lemma \ref{3-pointNonEq}. Since $det(D_b) = N_{\textbf{F}/\textbf{L}_p}(b) \in \textbf{L}_p$ has norm $\xi^6$ over $\textbf{k}$, the norm over $\textbf{k}$ of $\frac{\xi^2}{g(det(D_b))}$ is $1$ and therefore, using Hilbert's Theorem 90, we find a $\tilde \lambda \in L_p$ such that $\frac{\tilde \lambda}{g(\tilde \lambda)} = \frac{\xi^2}{g(det(D_b))}$. Now we get:
\[(\tilde \Sigma)^{g'} = \tilde \lambda B'^{g'} \Sigma (B^{-1})^{g'} = \tilde \lambda B' (123) \Sigma (123)^{-1} B^{-1} = \tilde \lambda B' (123) (123)^{-1} \Sigma B^{-1} = \tilde \Sigma\]
Which proves, that $\tilde \sigma$ is defined over $\textbf{L}_p$. To prove the property from Theorem \ref{AffineRep} (since $\tilde \sigma$ is surely $\textbf{F}$-birational it is also sufficent) we look at:
\[\xi^{-1}(\tilde \Sigma)^g =g(\tilde \lambda)\xi^{-1} B'^g \Sigma (B^{-1})^g =g(\tilde \lambda) \xi^{-1} (\xi A_g^{-1} B' D_{b^{-1}}^g) \Sigma (\xi (D_{b}^{-1})^g B^{-1} A_g^{op})\] \[ = g(\tilde \lambda) \frac{\xi^2}{g(det(D_b))} A_g^{-1} B' \Sigma B^{-1}A_g^{op} = A_g^{-1}(\tilde \Sigma)A_g^{op}\]
by Lemma \ref{3-pointNonEq}. This shows that $\tilde \Sigma$ is an affine representant of a $3$-link from $S^{op}$ to $S$.
\end{proof}
\noindent
If we are looking for $A_3 \sigma A_2$ that sends the base points of $\tilde \sigma^{-1}$ to the ones of $\tilde \sigma$ we can look at it as $B^{-1} A_3 \sigma A_2 B'$ that fixes $\{ [1:0:0], [0:1:0], [0:0:1]\}$. This leads us to $A_3 := B D_h ((A_2B')^{\sigma})^{-1}$, where $D_h$ is a diagonal. We want to find $D_h$ such that $A_3$ is the affine representant of an element of $Aut_{\textbf{k}}(S^{op})$.
\begin{lemma} \label{A3}
    Let $\textbf{L}_p, \textbf{L}_q,\textbf{F}$ be as in our Convention \ref{con:2SF}, where we assume $\textbf{L}_q$ is a degree $3$ extension of $\textbf{k}$. There is a $b \in \textbf{F}$ such that $N_{\textbf{F}/\textbf{L}_q}(b) = \xi^2$ and $B:= B(b), B'.= B(b^{-1})$. Let $N$ be the affine representant of a $\textbf{k}$-automorphism of $S$. We can find a $\lambda_b \in \textbf{F}$ such that for $D_h$ being the diagonal matrix with $\lambda_b, g'(\lambda_b), g'(g'(\lambda_b))$ on its diagonal (where $g'$ is as in Convention \ref{con:2SF}) $M := B D_h ((NB')^{\sigma})^{-1}$ is a affine representation of a $\textbf{k}$-automorphism of $S^{op}$.
\end{lemma}
\begin{proof}
    The existence of $b$ is due to Lemma \ref{3-pointNonEq}. The norm of $\frac{g(b)^3}{\xi^2}$ over $\textbf{L}_q$ is $1$ and by the use of Hilbert's Theorem 90 we can pick $\lambda_b \in \textbf{F}$ such that $\frac{\lambda_b}{g(\lambda_b)} = \frac{g(b)^3}{\xi^2}$. We pick as $D_h$ the diagonal with $\lambda_b, g'(\lambda_b),g'(g'(\lambda_b))$ on the diagonal. We then find that:
    \[M^{g'} = B (123) D_h^{g'} ((N B' (123))^{\sigma})^{-1} = B (123) D_h^{g'} (123)^{-1} ((NB')^{\sigma})^{-1} = B D_h ((NB')^{\sigma})^{-1} = M\]
    Which again shows, that $M$ is defined over $\textbf{L}_p$. For the representation property we look at:
    \[M = \xi^2 A_g^{op} B^g (D_b^{-1})^g D_h (D_{b^{-1}}^2)^g (((NB')^{\sigma})^{-1})^g (A_g^{op})^{-1} \] \[= A_g^{op} B^g D_h^g (((NB')^{\sigma})^{-1})^g (A_g^{op})^{-1} = A_g^{op} M^g (A_g^{op})^{-1}\]
    since $(D_b^{-1})^g D_h (D_{b^{-1}}^2)^g = \xi^{-2}D_h^g$ by our choice of $D_h$ and Lemma \ref{3-pointNonEq}. Thus $M$ is an affine representant of an automorphism in $Aut_{\textbf{k}}(S^{op})$.
\end{proof}
\noindent We are now again going to look at one example for the relation and use this to do our proof. We are going to use the construction of Lemma \ref{LargeCalc}, but this time we need some adjustments since we work with $\tilde \sigma$ and not all scalars are defined over $\textbf{k}$.
\begin{lemma} \label{HardRel}
    Let $\textbf{L}_p, \textbf{L}_q,\textbf{F}$ be as in our Convention \ref{con:2SF}, where we assume $[\textbf{L}_q: \textbf{k}] = 3$. Let $b \in \textbf{F}$ such that $N_{\textbf{F}/\textbf{L}_q}(b) = \xi^2$ and $B:= B(b), B'.= B(b^{-1})$. Let $\tilde \sigma$ be like in Lemma \ref{BB'}. We define the following:
    \[A_1 := I_3, \, A_{2i} := D_{\xi^{-1}} (A_{2i-1}^{\tilde \sigma})^{-1} = \tilde \lambda^{-1} D_{\xi^{-1}} M_2(B^{-1} A_{2i - 1})B'^{-1}, \, i = 1,2\] \[ A'_{6} := D_{\xi^{-1}} (A_{5}^{\tilde \sigma})^{-1} = \tilde \lambda^{-1} D_{\xi^{-1}} M_2(B^{-1} A_{5})B'^{-1}\] and \[A_{2i + 1} := B D_h ((A_{2i}B')^{\sigma})^{-1} = B D_h M_2(A_{2i}B'), \, i=1,2\]
    Then for $\tau$ the $3$-link represented by $\tilde \sigma$, $\alpha_i$ the automorphism represented by $A_i$ and $\alpha_6$ the automorphism represented by $A_6 := \tilde \lambda^{-3}D(B^{-1})^{-1}D_{\xi}^3 A'_6$ we get $\alpha_6 \tau \dots \alpha_1 \chi = id$ and this is a elementary relation. 
\end{lemma}
\begin{proof}
    We start by proving that all the matrices are affine representants. By Lemmas \ref{FixMatr} and \ref{A3} the $A_i, A'_6$ are affine representants. We need to prove, that $ \tilde \lambda^{-3} D(B^{-1})^{-1}D_{\xi}^3$ is an affine representant too. We can see that $D(B^{-1})$ is given by the coefficients of $B^{-1}$ like in Lemma \ref{LargeCalc}. We have $(B^{-1})^{g'} = (123)^{-1}B^{-1}$, which is a permutation of the lines of $B^{-1}$. But using this we see, that $D(B^{-1})^{g'} = D(B^{-1})$. Thus using that $\tilde \lambda \in \textbf{L}_p$ we find that $ \tilde \lambda^{-3} D(B^{-1})^{-1}D_{\xi}^3$ is defined over $\textbf{L}_p$. Since $(B^{-1})^g = \xi (D_b^{-1})^g B^{-1}A_g^{op}$ and $D(A)_{ii} = det(A)^6\frac{\prod_{k=1}^3(A^{-1})_{ik}^2}{\prod_{l=1}^3a_{li}}$ we can directly calculate that: 
    \[D(B^{-1})^g_{11} = D((B^{-1})^g)_{11} = D(\xi (D_b^{-1})^g B^{-1}A_g^{op})_{11} = D(B^{-1})_{22} (\xi^3 det((D_b^{-1})^g)^3) = D(B^{-1})_{22} \frac{\xi^{-3} \tilde \lambda^3}{g(\tilde \lambda ^3)}\]
    \[D(B^{-1})^g_{22} = D((B^{-1})^g)_{22} = D(\xi (D_b^{-1})^g B^{-1}A_g^{op})_{22} = D(B^{-1})_{33} (\xi^3 det((D_b^{-1})^g)^3) =D(B^{-1})_{33} \frac{\xi^{-3} \tilde \lambda^3}{g(\tilde \lambda ^3)}\]
    \[D(B^{-1})^g_{33} = D((B^{-1})^g)_{33} = D(\xi (D_b^{-1})^g B^{-1}A_g^{op})_{33} = D(B^{-1})_{11} (\xi^{12} det((D_b^{-1})^g)^3) = D(B^{-1})_{11} \frac{\xi^{6} \tilde \lambda^3}{g(\tilde \lambda ^3)}\]
    since $(\frac{\tilde \lambda}{g(\tilde \lambda)})^3 = \frac{\xi^6}{det(D_b^g)^3}$. We thus get that $\tilde \lambda^{3} D_{\xi}^{-3} D(B^{-1})$ is an affine representant. \\ Now we want to see that $A_6 \tilde \sigma A_5 \sigma \dots A_1 \sigma = id$. For this we relate the $A_i$ to $M_i(B^{-1})$. To simplify the calculation we will work with 
    \[C_{2i} := A_{2i} B', \, i = 1,2 \, , C_{2i + 1} := B^{-1} A_{2i + 1}, \, i = 0,1,2, \, C_{6} := A'_{6} B' \]
    for which we get:
    \[C_{2i} = \tilde \lambda^{-1} D_{\xi^{-1}} M_2(C_{2i -1}), \, i = 1,2, \, C_{6} = \tilde \lambda^{-1} D_{\xi^{-1}} M_2(C_{5})\] 
    \[ C_{2i + 1}  = D_h M_2(C_{2i}), \, i=1,2, \, C_1 = B^{-1}\] 
    Thus we observe with Remark \ref{RemDiag}, that in $Bir_{\textbf{F}}(\mathbb{P}^2)$: 
    \[A_6 \tilde \sigma A_5 \sigma \dots A_1 \sigma = (D(B^{-1})^{-1}D_{\xi}^3) A'_6 B' \sigma B^{-1} A_5 \dots B^{-1}A_1 \sigma \] \[= (D(B^{-1})^{-1}D_{\xi}^3) C_6\sigma \dots C_1 \sigma = (D(B^{-1})^{-1}D_{\xi}^3 D_{\xi^{-1}}) (D_{\xi^{-1}}^{-2}) (D_{\xi^{-1}}^4) D(B^{-1}) = id \] 
    By Lemma \ref{ElemRel} and using $\varphi, \varphi^{op}$, $\alpha_6 \tau \dots \alpha_1 \chi = id$ is an elementary relation.
\end{proof}
\begin{theorem}
    Take an elementary relation of $S$, where both $3$-point classes are non-equivalent and their splitting fields are degree $3$ extentions of $\textbf{k}$. Then the determinant of the elementary relation is trivial. 
\end{theorem}
\begin{proof}
 By Corollary \ref{DepPoint} it is enough to look at all possible base points of $\tau \alpha_1$ and choose one version of the other $5$ automorphisms such that the base point are send to each other by the links as it is shown in Figure \ref{fig: rel}. We do this by choosing $B,B', \tilde \sigma$ as in Lemma \ref{BB'}. Since the $3$-points occuring in this relation cannot be on one of the $3$ lines $x,y,z=0$ we can get all $3$-points non equivalent to $p$ with $B$. We thus can use the example relation from Lemma \ref{HardRel}. We now need to calculate the $A_i$ explicitely to be able to find the determinant, since in this case a priori we are only in $(\textbf{F}^*)^3$. To calculate the $A_i$ we use the $C_i$ from the proof of Lemma \ref{HardRel}. Observe, that for a diagonal matrix $D$, a general matrix $A$ and a scalar $\mu$ we get as in Remark \ref{RemDiag}, that: 
    \[M_2(A D) = D^{-2} M_2(A), \] 
    \[ M_2(DA) = det(D)^{-1} M_2(A) D\] 
    \[M_2(\mu A) = \mu^{-2} M_2(A)\]
    Using this and $M_2 \circ M_2 = M_3$ we get that (writing $M_i$ instead of $M_i(B^{-1})$):
    \[C_2 =  \lambda_2 D_{\xi^{-1}} M_2, \, \lambda_2 = \tilde \lambda^{-1}\]
    \[C_3 =  \lambda_3 D_h M_3 D_{\xi^{-1}}, \, \lambda_3 = \tilde \lambda^2\]
    \[C_4 =  \lambda_4 D_{\xi} M_4 D_h, \, \lambda_4 = \tilde \lambda^{-5} det(D_h^{-1})\]
    \[C_5 = \lambda_5 D_h^{-1} M_5 D_{\xi}, \, \lambda_5 = \tilde \lambda^{10} det(D_h)^2\]
    \[C_6 =   \lambda_6 D_{\xi^{-3}} M_6 D_h^{-1}, \, \lambda_6 =\tilde \lambda^{-22} det(D_h)^{-3}\] 
 The only thing left to do is to calculate the determinant and check that it actually lies in $(\textbf{k}^*)^3$ and not only in $(\textbf{F}^*)^3$. \\
 We can see using Lemma \ref{LargeCalc}, that: \[det(\tilde \lambda^3 D(B^{-1})D_{\xi}^{-3}) = \tilde \lambda^9 \xi^{-9} det(D(B^{-1})) = \tilde \lambda^9 P(B^{-1})^{-1} P(B)^2 det(B^{-1})^{18}\] Our goal is to see that \[ \frac{det(A_2A_4A'_6)det(A_1A_3A_5)^{-1}} {det(\tilde \lambda^{3}D(B^{-1}) D_{\xi}^{-3})} = det(A_2A_4A_6)det(A_1A_3A_5)^{-1}\]
 belongs to $(\textbf{k}^*)^3$. We have that:
 \[det(A_2A_4A'_6)det(A_1A_3A_5)^{-1} = det(C_2C_4C_6)det(C_1C_3C_5)^{-1} det(B')^{-3}det(B)^3 \] \[= d(B^{-1}) \tilde \lambda^{-117} det(D_h)^{-18}det(B')^{-3}det(B)^{-3}\]
 Using Lemma \ref{LargeCalc} we get that:
 \[\frac{det(\tilde \lambda^3 D(B^{-1})D_{\xi}^{-3})}{det(A_2A_4A'_6)det(A_1^{-1}A_3^{-1}A_5^{-1})} =  \frac{\tilde \lambda^9 P(B^{-1})^{-1} P(B)^2 det(B^{-1})^{18}}{d(B^{-1}) \tilde \lambda^{-117} det(D_h)^{-18}det(B')^{-3}det(B)^{-3}}\]
 \[= P(B^{-1})^{-15} P(B)^9 det(B^{-1})^{60} det((B^{-1})^{\sigma})^{42} \tilde \lambda^{126} det(D_h)^{18}det(B')^3det(B)^3\]
 We now want to see, that this lies in $(\textbf{k}^*)^3$. To do so, we use the following properties:
\[g(P(B^{-1})) = P(\xi (D_b^{-1})^gB^{-1}A_g^{op}) = \xi^6 g(det(D_b)^{-3}) P(B^{-1}) = (\frac{\tilde \lambda}{g(\tilde \lambda)})^3P(B^{-1})\]
Where the first equality is by interchanging $P,g$ and using Lemma \ref{3-pointNonEq} on $B$ and the second one is a property of $P$. This implies, that 
\begin{equation} \label{eq:1P}
P(B^{-1}) \tilde \lambda^3 \in \textbf{k}^*
\end{equation}
and a similar calculation shows
\[P(B)\tilde \lambda^{-3} \in \textbf{k}^*\] 
which again implies 
\begin{equation} \label{eq:2Ps}
P(B^{-1})P(B)\in \textbf{k}^*
\end{equation}
We also see with Lemma \ref{3-pointNonEq} and $g(det(D_b))\xi^{-2} = \frac{g(\tilde \lambda)}{\tilde \lambda}$ by Lemma \ref{BB'}:
\[g(det(B)) = \xi^{-2} g(det(D_b)) det(B) = \frac{g(\tilde \lambda)}{\tilde \lambda} det(B)\]
which implies 
\begin{equation} \label{eq:1B}
det(B) \tilde \lambda^{-1} \in \textbf{k}^*
\end{equation}
and similarly 
\[det(B') \tilde \lambda \in \textbf{k}^*\] 
and thus also 
\begin{equation} \label{eq:2Bs}
det(B)det(B') \in \textbf{k}^*
\end{equation}
Since:
\[\frac{\lambda_b}{g(\lambda_b)} = \frac{g(b)^3}{\xi^2}, \, \frac{\tilde \lambda}{g(\tilde \lambda)} = \frac{\xi^2}{N_{\textbf{F}/\textbf{L}_p}(g(b))}\]
We can find that:
\[\frac{det(D_h)}{g(det(D_h))} = N_{\textbf{F}/\textbf{L}_p}(\frac{\lambda_b}{g(\lambda_b)}) = \frac{N_{\textbf{F}/\textbf{L}_p}(g(b))^3}{\xi^6} =(\frac{g(\tilde \lambda)}{\tilde \lambda})^3\]
We thus get:
\begin{equation} \label{eq:Dh}
    \tilde \lambda^3 det(D_h) \in \textbf{k}^*
\end{equation}
Now we use this to reduce the term:
\[\frac{det(\tilde \lambda^3 D(B^{-1})D_{\xi}^{-3})}{det(A_2A_4A'_6)det(A_1^{-1}A_3^{-1}A_5^{-1})}\] 
\[= P(B^{-1})^{-15} P(B)^9 det(B^{-1})^{60} det((B^{-1})^{\sigma})^{42} \tilde \lambda^{126} det(D_h)^{18}det(B')^3det(B)^3\]
From the fact that $A_2^{-1} = \tilde \lambda B' (B^{-1})^{\sigma} D_{\xi}, \, det(D_{\xi}) \in (\textbf{k}^*)^3, \, \tilde \lambda^{126} = (\tilde \lambda ^3)^{ 42}$ we can find:
\[= P(B^{-1})^{-15} P(B)^9 det(A_2^{-1})^{42} det(D_h)^{18}det(B')^{-39}det(B)^{-57}\]
From \eqref{eq:2Ps}, \eqref{eq:2Bs}, $39=3 \cdot 13$ and $ det(A_2^{-1})^{42} \in (\textbf{k}^*)^3$ we get:
\[= P(B^{-1})^{-24} det(D_h)^{18}det(B)^{-18}\]
From \eqref{eq:1P} we get, since $72 = 24 \cdot 3$
\[= (\tilde \lambda)^{72} det(D_h)^{18}det(B)^{-18}\]
Using \eqref{eq:Dh} and $72 = 18 + 54 = 18 + 3 \cdot 3 \cdot 6$
\[= (\tilde \lambda)^{18} det(B)^{-18} \]
Using \eqref{eq:1B} we finally get:
\[ = 1\]
\end{proof}
\noindent
Now if one or both of the $3$-points have a splitting field, that is a degree $6$ extension, we look at how the argument can be reduced to the previous cases using Galois theory.
\begin{theorem} \label{DetElRef}
    For all elementary relations the determinant is trivial.
\end{theorem}
\begin{proof}
     We take an elementary relation over $S$, with $p,q$ the two $3$-points of the relation. Let $\textbf{L}_p, \textbf{L}_q$ be their splitting fields (we choose the same field if they are equivalent). If we look at $([\textbf{L}_p:\textbf{k}], [\textbf{L}_q:\textbf{k}])$ we have proven the theorem in the $(3,3)$ case. If we have the $(3,6)$ case we can invert the relation to switch $p,q$. Thus we are left with looking at $(6,3),(6,6)$. 
    \newline
    We start with $\textbf{(6,3)}$. Since $\textbf{k} \subset \textbf{L}_p$ is Galois of degree $6$ we can find a field $\textbf{L}_2$ such that $\textbf{k} \subset \textbf{L}_2 \subset \textbf{L}_p$, where the first extension is of degree $2$ and the second of degree $3$. Assume $S_{\textbf{L}_2}$ would have $\textbf{L}_2$-points, then they would correspond to a $2$-point in $S_{\textbf{k}}$, since it is just the Galois action of this field extention, which is of degree $2$. But this is impossible, thus $S_{\textbf{L}_2}$ is a non-trivial Severi-Brauer surface over a perfect field. Now $q$ must still be a $3$-point (its order can not become higher through a field extension and not lower than $3$). We just need to make sure, that it still has a splitting field with field extension $3$. But $q$ splits over $\textbf{L}_2 \textbf{L}_q$ and by primitive element theorem $\textbf{k} \subset \textbf{L}_q$ is generated by one element with minimal polynomial over $\textbf{k}$ of degree $3$. Using the same element this still holds true for $\textbf{L}_2 \subset\textbf{L}_2 \textbf{L}_q$. Now we are in the $(3,3)$ case and can use the previous result. We can use the isomorphism $\varphi, \varphi^{op}$ for $\xi,\xi^{-1}$ from Lemma \ref{3pLem} that we have always used, since for degree $6$ extensions they give the same $A_g$ and we can keep our affine representation of the automorphisms. Thus we get, by the theorems on the $(3,3)$ cases, that the determinant of the word is in $\textbf{k}^* \cap (\textbf{L}_2^{*})^3$. But it is just left to see now, that $\textbf{k}^* \cap (\textbf{L}_2^{*})^3 \subset (\textbf{k}^{*})^3$. For this pick $a \in \textbf{k}^*$ and $b \in \textbf{L}_2^*$ such that $a = b^3$. Then if the polynomial $x^3 - a$ has a solution in $\textbf{k}$ we are done. If not, we have that the minimal polyonmial of $b$ over $\textbf{k}$ is of degree $3$, which is impossible for a degree $2$ field extension. This finishes this case.
    \newline
    For $\textbf{(6,6)}$ we can take again $\textbf{L}_2$ such that $\textbf{k} \subset \textbf{L}_2 \subset \textbf{L}_p$ with the same degrees as above. Like above $S_{\textbf{L}_2}$ is still a non-trivial Severi-Brauer surface over a perfect field and $q$ is still a $3$-point (its splitting field can be of both degree $3$ or $6$). Now we are either in the $(3,3)$ case and like above using the same $\varphi, \varphi^{op}$ and $\textbf{k}^* \cap (\textbf{L}_2^{*})^3 \subset (\textbf{k}^{*})^3$ we are finished by the previous result. In the $(3,6)$ case we use again the same $\varphi, \varphi^{op}$ and see by the previous case, that the determinant of the word is in $ (\textbf{L}_2^{*})^3$. But due to using the same isomorphisms, we see it is also in $\textbf{k}^*$. Now we argue like in the end of the $(6,3)$ case.  
\end{proof}
\noindent
Since both the elementary and trivial relations are mapped to the identity by the determinant map, we can prove the following result.
\begin{corollary}\label{DetOnBir}
    For every choice $R$ of representators of the different classes of $\textbf{k}$-links (from $S$ to $S^{op}$) there is a well defined non-trivial groupoid homomorphism \[det_R: G_S \rightarrow \textbf{k}^*/(\textbf{k}^*)^3\] \[R \mapsto 1\] \[\alpha \in Aut_{\textbf{k}}(S) \mapsto det(\alpha)\] \[\beta \in Aut_{\textbf{k}}(S^{op}) \mapsto det(\beta)^{-1}\] Where $G_S := \{ \varphi: S_1 \dashrightarrow S_2 : S_1,S_2 \in \{ S, S^{op}\}, \, \varphi \text{ birational }\}$.
\end{corollary}
\begin{proof}
    From Theorem \ref{Isk} and Lemma \ref{Links36} we get that every $G_S$ is generated by links from $R$ and automorphisms of $S,S^{op}$, with elementary relations given by Lemma \ref{ElRel} and trivial relations given by $\chi^{-1} \chi = id, \, \chi^{-1} \beta \chi \alpha = id$. For $\varphi \in G_S$ given by $\alpha_{r+1}\prod_{i=1}^r\chi_i \alpha_i$, where $\chi_i \in R$ and $\alpha_i$ are automorphisms of $S$ if $i$ is odd and automorphisms of $S^{op}$ if $i$ is even, we can define $det_R(\varphi) := det(\alpha_1 \alpha_2^{-1} \dots )$. The first trivial relation is sent to the identity by definition, the second one by Lemma \ref{DetTrivRef}. The elementary relation is sent onto the identity by Theorem \ref{DetElRef}. To show that it is non trivial we take a $3$-point $p$ whose splitting field $\textbf{L}_p$ is a degree $3$ extension over $\textbf{k}$ (exists like mentioned before) and thus for taking the usual isomorphism $\varphi$ we get the usual element $\xi$ which is not a norm and thus not in $(\textbf{k}^*)^3$. This $\xi$ is the determinant of $A_g$ and we see $A_g^{-1} A_g A_g = A_g = A_g^g$ represents a $\textbf{k}$-automorphism of $S$ and thus $\xi$ is in the image but is not the identity. 
\end{proof}

\begin{corollary}[Follows from Theorem \ref{SBGHom} and Corollary \ref{DetOnBir}] \label{AbMap}
    For every non-trivial Severi-Brauer surface $S$ over a perfect field $\textbf{k}$ and $q$ a $3$-point of $S$, there is a group homomorphism \[\Phi:Bir_{\textbf{k}}(S) \rightarrow \bigoplus_{p \in \mathcal{E}_3 \setminus \{ q \}} \mathbb{Z}/3\mathbb{Z} \oplus \bigoplus_{p \in \mathcal{E}_6} \mathbb{Z} \oplus \textbf{k}^* / (\textbf{k}^*)^3\] which does not send all automorphisms to trivial elements.
\end{corollary}

\section{The abelianization of $Bir_{\textbf{k}}(S)$}
\noindent Now we are going to use the group homomorphism from the last chapter to find our abelianization. The last ingrediant that we need is the abelianization of the auromorphism group of $S$.
\begin{lemma}
    Let $\textbf{L}$ be a splitting field of a $3$-point of $S$, that is a degree $3$ extension over $\textbf{k}$ and $\xi$ as usual. Take the matrix $A_g$ as usual. The central simple algebra associated to $S_{\xi}$ is the following:
    \[D:= \{ A \in Mat_{3\times 3}(\textbf{L}) : A_g^{-1}AA_g = A^g\}\]
\end{lemma}
\begin{proof}
    Note that $D$ is a subring of $Mat_{3 \times 3}(\textbf{L})$. Identifying $\textbf{k}$ with $\textbf{k}I_3$ makes it a associate $\textbf{k}$ algebra.
     We can identify $\textbf{L}$ with the diagonal elements of the form $\begin{pmatrix}
         \lambda & 0 & 0 \\
         0 & g(\lambda) & 0 \\
         0 & 0 & g^2(\lambda) 
     \end{pmatrix}$ for $\lambda \in \textbf{L}$. Now we see, that $\{ I_3, A_g, A_g^2\}$ is a basis of the $\textbf{L}$-vector space $D$. Thus for a $b \in \textbf{L}$ where $\textbf{L} = \textbf{k}(b)$ we get, that $1,b,b^2$ is a basis of $\textbf{L}$ over $\textbf{k}$ thus $\{ I_3, A_g, A_g^2, bI_3, bA_g, bA_g^2,b^2I_3,b^2A_g,b^2A_g^2\}$ is a $\textbf{k}$ basis, where we again identify $b$ in $D$ as $\begin{pmatrix}
         b & 0 & 0 \\
         0 & g(b) & 0 \\
         0 & 0 & g^2(b) 
     \end{pmatrix}$. Thus we see, that the dimension is $9$. We are left to show, that it is a central simple division algebra over $\textbf{k}$ that is associated to $S_{\xi}$.
     \newline
     \underline{The center of $D$}: As $\lambda I_3, \, \lambda \in \textbf{k}$ commutes with all matrices we have $\textbf{k} \subset Z(D)$. Conersely, let now $A \in Z(D)$. The first column of $Ab$ is given by $ba_{11}, g(b)a_{21}, g^2(b)a_{31}$. The first column of $bA$ is given by $a_{11}b, a_{21}b, a_{31}b$ and since $g(b) \neq b$ we get, that $a_{21}=a_{31}=0$ thus $A$ is diagonal with diagonal $c,g(c),g^2(c), \, c \in \textbf{L}$. Using $I_3 + A_g + A_g^2$ we find $c=g(c)$ and thus $c \in \textbf{k}$. Thus $D$ is central.
     \newline
     \underline{Division Algebra}: Let $A \in D$ and $a_1$ its first column. We find if $A \neq 0$, that $a_1 \neq 0$ and thus the columns of $A$ correspond to the $3$-point $a_1, A_g(g(a_1)), A_g(g^2(A_1))$ and are thus not collinear (\cite[Proof of 2.3.2]{BSY}). Thus $A$ is invertible and thus $D$ is a division algebra.
     \newline
     \underline{Simple}: $A$ is simple, because $D \setminus D^* = \{ 0 \}$. 
     \newline
     \underline{\textbf{L} is a splitting field}: We take $D':= D \otimes_{\textbf{k}} \textbf{L}$ and the following homomorphism:
     \[f: D' \rightarrow Mat_{3 \times 3}(\textbf{L}), \, A \otimes c \mapsto cA\]
     (The multiplication with $\textbf{L}$ being the normal one now). It is now left to see, that \newline 
     $\{ I_3, A_g, A_g^2, bI_3, bA_g, bA_g^2,b^2I_3,b^2A_g,b^2A_g^2\}$ is linearly independent over $\textbf{L}$. Define $b_{ij}:= b^i A_g^j, \, i,j=0,1,2$.  Now take $\lambda_{ij} \in \textbf{L}$ such that $\sum_{i,j=0}^2 \lambda_{ij}b_{ij} = 0$. To prove, that $\lambda_{ij} = 0$ it is enough to see, that if we have $a_i \in \textbf{L}, \, i=0,1,2$ such that $a_0 + a_1g^j(b) + a_2g^j(b^2) = 0, \, j=0,1,2$ we have $a_0=a_1=a_2=0$. From the three equations we gind, that $a_1(g(b)-b) + a_2(g(b^2)-b^2)= 0$ and $a_1(g^2(b)- b) + a_2(g^2(b^2) - b^2) = 0$. If one of the $a_i$ is zero all of them are. Thus we assume none of them are $0$ and thus find $- \frac{a_3}{a_2} = \frac{g(b) - b}{g(b)^2 - b^2} = \frac{1}{g(b) + b}$ and $- \frac{a_3}{a_2} = \frac{g^2(b) - b}{g^2(b)^2 - b ^2} = \frac{1}{g^2(b) + b}$. This implies that $g(b) = g^2(b)$, which is false. Thus $b_{ij}$ are linearly independant in $Mat_{3 \times 3}(\textbf{L})$ and thus a basis. Thus the above homomorphism is a isomorphism.
     \newline
     \underline{Cocyle of $D$}: The corresponding cocycle of $S_{\xi}$ in $H^1(Gal(\textbf{L}/\textbf{k}), PGl_3(\textbf{L}))$ is given by $A_g$. Now to find the one for $D$ we want a matrix $a_g$ such that $f \circ g = a_g \circ g \circ f$. We observe, that $f(g(A \otimes \lambda)) = g(\lambda) A = g(\lambda) A_g A^g A_g^{-1} = [A_g] (g (f (A \otimes \lambda))) $. Thus $S_{\xi}$ is the Severi-Brauer surface, that belongs to $D$ (3.6 \cite{jahnel}).
\end{proof}

\begin{lemma} \label{AutAbel}
    Let $S$ be a non-trivial Severi-Brauer surface. The determinant map is the abelianization of $Aut_{\textbf{k}}(S)$.
\end{lemma}
\begin{proof}
    This follows, since for $\varphi$ as usual we get $\varphi Aut_{\textbf{k}}(S) \varphi^{-1} = D^*/\textbf{k}^*$ for the $D$ from above. Using the isomorphism $f$ one can define the reduced norm of the central simple algebra as the composition of the determinant on $Mat_{3 \times 3}(\textbf{L})$ with $f$ (for a more accurate description see \cite{Gille_Szamuely_2006} Chapter 2.6). Thus in our case we find the reduced norm on $D$ is the determinant. Using \cite{Wang} we get, that the abelianization of $D$ is the determinant. Now the abelianization of $D^*/\textbf{k}^*$ is also the determinant mapping to $\textbf{k}^*/(\textbf{k}^*)^3$, since any matrix with determinant in $(\textbf{k}^*)^3$ can be represented by a matrix with determinant $1$, which is a multiplication of commutators in $D$ and then its class is a multiplication of the commutators of the classes of those commutators in $D^*/\textbf{k}^*$. 
\end{proof}
\begin{lemma} \label{ImageDET}
    Let $S$ be a non-trivial Severi-Brauer surface. Then $det(Aut_{\textbf{k}}(S)) = det(Aut_{\textbf{k}}(S^{op}))$.
\end{lemma}
\begin{proof}
    Take $\alpha \in Aut_{\textbf{k}}(S)$. By Lemma \ref{Fixpoints} there is a $3$-point $p$ which is fixed by $\alpha$. Take $\chi: S \dashrightarrow S^{op}$ a Sarkisov link based at $p$. By Lemma \ref{DetTrivRef} the automorphism $\beta := \chi^{-1} \alpha \chi \in Aut_{\textbf{k}}(S)$ satisfies $det(\alpha) = det(\beta)$.
\end{proof}
\noindent Finally we have all the tools to prove Theorem \ref{MainB}.

\begin{proof}[Proof of Theorem \ref{MainB}]
    Take $\Phi$ from Corollary \ref{AbMap} (which is the one described in Theorem \ref{MainB}). Since the commutator subgroup of $Bir_{\textbf{k}}(S)$ is in the kernel of $\Phi$, we can define a map $\Phi^{Ab}$ the map induced by $\Phi$ on the abelianization of $Bir_{\textbf{k}}(S)$. Now to study this maps kernel we take a $[f]$ in the kernel of this map. Using Theorem \ref{Isk} we can write: \[f = (\Pi_{i=1}^k \alpha_{2i - 1} \chi_{p_{2i-1}}^{-1} \alpha_{2i} \chi_{p_{2i}}) \alpha_{2k + 1} \] 
    \[= \Pi_{i=0}^k \alpha_{2i+1} \Pi_{i=1}^k \chi_{p_{2i -1}}^{-1} \chi_q \chi_q^{-1} \alpha_{2i} \chi_q \chi_q^{-1} \chi_{p_{2i}}\]
    \[ = (\Pi_{i=0}^k \alpha_{2i+1}) \chi_q^{-1} (\Pi_{i=1}^k \alpha_{2i}) \chi_q (\Pi_{i=1}^k B_{p_{2i -1},q}B_{q,p_{2i}})\] 
    where $p_i \in \mathcal{E}_3 \cup \mathcal{E}_6$, $\chi_p \in R$ are the representants of the Sarkisov links of class $p$, $B_{p_i,q} := \chi_{p_i}^{-1}\chi_q$ and $\alpha_i$ are automorphisms of $S$ if the index is uneven or $S^{op}$ if the index is odd. If there is $i$ even and $j$ odd with $p_i \sim p_j$ we can find $B_{q,p_i} B_{p_j,q} = B_{q,q} = id$, we can thus substitute. Since $B_{q,q} = id$ we can assume $p_i \not\sim q$. If there is a class $p$, that appears more than 3 times we can find $\beta_1, \dots , \beta_6$ such that $\beta_1 \chi_p^{-1} \beta_2\chi_q \dots \beta_6 \chi_q = id$ and thus $B_{p,q}^3 = \beta_1 \beta_3 \beta_5 \chi_q^{-1} \beta_2 \beta_4 \beta_6 \chi_q$. Using this we can assume that $f$ has the following form:
    \begin{equation}
        f = \alpha \chi_q^{-1} \alpha' \chi_q \Pi_{i=1}^r B_{p_i,q}^{a_i}
    \end{equation}
    For $p_i \in \mathcal{E}_3 \cup \mathcal{E}_6 \setminus \{ q \}$ non equivalent and $a_i>0$, where we have:
    \begin{equation} 
        \forall 1 \leq i \leq r \text{ such that } p_i \in \mathcal{E}_3, \, a_i \leq 2
    \end{equation}
    where $\alpha \in Aut_{\textbf{k}}(S), \alpha' \in Aut_{\textbf{k}}(S^{op})$. Now since $f$ is in the kernel of $\Phi^{Ab}$ we find, that all $r = 0$. We can now take a $3$-point $q$ that is the fixpoint of $\alpha'$ (see Lemma \ref{Fixpoints}) and see, that in the abelianization we have $\alpha \chi_p^{-1} \alpha' \chi_p = (\alpha \chi_p^{-1} \chi_q) (\chi_q^{-1}\alpha' \chi_p) = \alpha (\chi_p^{-1} \chi_q) \tilde \alpha' (\chi_q^{-1} \chi_p) = \alpha \tilde \alpha'$, where $\tilde \alpha' \in Aut_{\textbf{k}}(S)$ such that $\chi_q^{-1} \alpha' = \tilde \alpha' \chi_q^{-1}$ (exists since $\alpha'$ fixes $q$). Finally we see, that the automorphisms of determinant $1$, are commutators of automorphisms and thus also commutators in $Bir_{\textbf{k}}(S)$. Thus $\Phi^{Ab}$ is injective and thus its image is the abelianization of $Bir_{\textbf{k}}(S)$. We also note, that the image contains $\bigoplus_{p \in \mathcal{E}_3 \setminus \{ q \}} \mathbb{Z}/3\mathbb{Z} \oplus \bigoplus_{p \in \mathcal{E}_6} \mathbb{Z}$ and that DET is the set of all the products of determinants of automorphisms of $S,S^{op}$ and thus by Lemma \ref{ImageDET} it is the image of $Aut_{\textbf{k}}(S)$ under the determinant. 
\end{proof}
\noindent
 To determine DET more accurately is a more complicated task. Every automorphism of $S,S^{op}$ has a fixpoint that is a $3$-point and over that points splitting field the automorphism is represented as a diagonal matrix. If that points splitting field has degree $3$ over $\textbf{k}$ the value of the determinant can be all values in the image of the norm. If it is of degree $6$ we need further restrictions on the diagonal. Another more algebraic way to look at this is to see, that the $\textbf{k}^*/(\textbf{k}^*)^3$ part is generated by the images of the norm map of the central simple algebras of $S,S^{op}$ modulo $(\textbf{k}^*)^3$. Thus the remaining open question is:
\begin{question}
    What is the exact image of $Bir_{\textbf{k}}(S)$ under the group homomorphism $\Phi$ from Corollary \ref{AbMap}? 
\end{question}
\noindent As a last result we can use this group homomorphisms to find some maximal subgroups.
\begin{corollary}
    Take $q$ a $3$-point of $S$. Take $p \in \mathcal{E}_3 \setminus \{ q \} \cup \mathcal{E}_6$ and $\Phi_{p}$ the composition of the $\Phi$ (defined with $q$) from Corollary \ref{AbMap} with the projection on the $p$ part of the direct sum. Define $M_{p} := \Phi_p^{-1}(M'_p)$, where $M'_p = \{ 0_p \} $ if $p$ is a $3$-point and $M'_p = a_p\mathbb{Z}$ where $a$ is a prime number if it is a $6$-point, which are the maximal subgroups of the respective group. Then $M_p$ is a maximal subgroup of $Bir_{\textbf{k}}(S)$.
\end{corollary}
\begin{proof}
    We know $M_p$ contains all Blocks of the form $B_{p',q}, \, p' \in \mathcal{E}_3 \setminus \{ p\} \cup \mathcal{E}_6 $ like above and $B^{\beta} := \chi_q^{-1} \beta \chi_q, \, \beta \in Aut_{\textbf{k}}(S)$. It also contains all elements from $Aut_{\textbf{k}}(S)$. Thus the only missing block to generate $Bir_{\textbf{k}}(S)$ is $B_{q,p}$. Now take $\psi \in Bir_{\textbf{k}}(S) \setminus M_p$. Then take $l:= \Phi_p(\psi)$ (in the $3$-point case, we take it to be either $1,-1$). Thus we get, that $\psi B_{q,p}^{-l} \in ker(\Phi_p) \subseteq M_p$ thus $\langle \psi , M \rangle = \langle B_{q,p}^l,M \rangle$ and we can assume $\psi = B_{q,p}^l$. By possibly inverting $\psi$ we can assume that $l > 0$. Thus in the $3$-point case $l=1$ and we are finished. In the $6$-point case $l$ is not devided by $a$, thus they are coprime and we get from Bézouts Lemma $m,n \in \mathbb{Z}$ such that $ml + na = 1$. Thus $Bir_{\textbf{k}}(S) \supseteq \langle \psi,M_p \rangle \supseteq \langle B_{q,p},M \rangle = Bir_{\textbf{k}}(S)$, which finishes our proof.
\end{proof}
\bibliographystyle{alphadin}
\bibliography{refs}
\end{document}